\documentclass[12pt]{article}
\usepackage{amssymb}
\usepackage{amsfonts}
\usepackage{theorem}
\usepackage{graphicx}
\usepackage[leqno]{amsmath}
\usepackage{enumerate}
\usepackage{indentfirst}

\newenvironment{proof}[1][Proof]{\textbf{#1.} }{\hfill $\square$}

\newtheorem{lem}{Lemma}
\newtheorem{thm}{Theorem}
\newtheorem{defin}{Definition}
\newtheorem{rem}{Remark}
\newtheorem{prop}{Proposition}
\newtheorem{corol}{Corollary}

\newcommand{\R}{\mathbb{R}}

\newcommand{\N}{\mathbb{N}}

\newcommand{\spas}{\mathcal{S}}
\newcommand{\spam}{\mathcal{M}}
\newcommand{\spaprod}{\mathcal{B}}

\newcommand{\Prb}{\mathbb{P}}

\newcommand{\E}{\mathbb{E}}
\newcommand{\tri}{\mathcal{F}}
\newcommand{\al}{\alpha}

\newcommand{\eps}{\varepsilon}

\newcommand{\D}{\Delta}
\newcommand{\ind}{\mathbf{1}}

\newcommand{\sgn}{\mbox{sgn}}
\newcommand{\ka}{\kappa}

\title{$L^p$-Solutions for Reflected Backward Stochastic Differential Equations}
\author{Sa\"id Hamad\`ene and  Alexandre Popier \\Laboratoire de statistiques et processus\\
    Universit\'e du Maine\\
    Avenue Olivier Messiaen\\
    F-72085 Le Mans Cedex 9\\
    France}

\begin{document}

\maketitle

\begin{abstract}
This paper deals with the problem of existence and uniqueness of a solution for a backward stochastic differential equation (BSDE for short) with one reflecting barrier in the case when the terminal value, the generator and the obstacle process are $L^p$-integrable with $p\in ]1,2[$. To construct the solution we use two methods: penalization and Snell envelope. As an application we broaden the class of functions for which the related obstacle partial differential equation problem has a unique viscosity solution.
\end{abstract}



\section{Introduction}
%-------------------

The notion of nonlinear backward stochastic differential equation (BSDE for short) was introduced by Pardoux and Peng (1990, \cite{ppeng}). A solution of this equation, associated with a terminal value $\xi$ and a {\it generator or coefficient} $f(t,\omega,y,z)$, is a couple of adapted stochastic processes $(Y_t,Z_t)_{t\leq T}$ such that:
\begin{equation}\label{standbsde}
\forall t \leq T, \quad Y_t=\xi+\int_t^Tf(s,Y_s,Z_s)ds-\int_t^TZ_sdB_s,
\end{equation} 
where $B$ is a Brownian motion and adaptation is related to the completed filtration of $B$.
\medskip

In their seminal work \cite{ppeng}, Pardoux and Peng proved the existence and uniqueness of a solution under suitable assumptions, mainly square integrability of $\xi$ and the process $(f(t,\omega,0,0))_{t\leq T}$, on the one hand, and, the Lipschitz property w.r.t. $(y,z)$ of the generator $f$, on the other hand. Since this first result, it has been widely recognized that BSDE's provide a useful framework for formulating a lot of mathematical problems such as problems in mathematical finance (see e.g.\cite{BElKa, KPQ97,KPQ, nicolerouge},... ), stochastic control and differential games (see e.g. \cite{HLP95,HLP295},...), partial differential equations (see e.g.\cite{pardoux,ppeng2},... ) and so on.
\medskip

Another types of BSDEs, the one barrier reflected BSDEs, have been introduced by El-Karoui et al. \cite{KPQ97}. In the framework of those BSDEs, one of the components of the solution is forced to stay above a given {\it barrier/obstacle} process $(L_t)_{t\leq T}$. Therefore a solution is a triple of adapted processes $(Y_t,Z_t,K_t)_{t\leq T}$ which satisfies: 
\begin{eqnarray} \label{RBSDE} 
&&Y_t = \xi + \int_t^T f(s,Y_s,Z_s) ds + K_T - K_t - \int_t^T Z_s dB_s, \quad 0 \leq t \leq T,\\ \nonumber 
&&Y_t \geq L_t, \quad 0 \leq t \leq T \mbox{ and }\int_0^T(Y_s-L_s)dK_s=0., \ P-\mbox{a.s.};
\end{eqnarray}
here the process $K$ is non-decreasing and its role is to push upwards $Y$ in order to keep it above the obstacle $L$.

These types of equations are connected with a wide range of applications especially the pricing of America options in markets constrained or not, mixed control, partial differential variational inequalities, real options (see e.g. \cite{hamdjepop, hamela,KPQ97, KPQ, hamjea,hamzha,lepeltier1},...and the references therein). Once more under square integrability of the data and Lipschitz property of the coefficient $f$, the authors of \cite{KKPPQ97} show existence and uniqueness of the solution.
\medskip

For several years there have been a lot of works which deal with the issue of existence/uniqueness results under weaker assumptions than the ones of Pardoux-Peng \cite{ppeng} or El-Karoui et al \cite{KKPPQ97}. However, for their own reasons, authors focus only on the weakness of the Lipschitz property of the coefficient and not on square integrability of the data $\xi$ and $(f(t,\omega,0,0))_{t\leq T}$. Actually there have been relatively few papers which deal with the problem of existence/uniqueness of the solution for BSDEs in the case when the coefficients are not square integrable. Nevertheless we should point out that El-Karoui et al. (1997, \cite{KPQ97}) and Briand et al. (2003, \cite{briand}) have proved existence and uniqueness of a solution for the standard BSDE (\ref{standbsde}) in the case when the data belong only to $L^p$ for some $p\in ]1,2[$. Therefore the main objective of our paper is to complete those works and to study the reflected BSDE (\ref{RBSDE}) in the case when the terminal condition $\xi$ and the generator $f$ are only $p$-integrable with $p\in ]1,2[$. The main motivation of this work is that in several applications (finance, control, games, PDEs,...) the data are not square integrable and to assume them so is somehow restrictive.

In this article we show that if $\xi$, $\sup_{t\leq T}(L_t^+)$ and $\int_0^T|f(t,0,0)|dt$ belong to $L^p$ for some $p\in ]1,2[$, then the BSDE (\ref{RBSDE}) with one reflecting barrier associated with $(f,\xi,L)$ has a unique solution. We prove existence and uniqueness of the solution in using penalization and Snell envelope of
processes methods. We finally deal with the partial differential variational inequality (PDVI in short) associated with the reflected BSDE under consideration in the case when randomness comes from a standard diffusion process. Actually we show existence and uniqueness of the solution in viscosity sense for that PDVI in some appropriate space. This result broadens the class of functions for which there exists a unique solution for this associated PDVI with obstacle.
\medskip

The paper is organized as follows: the next section contains all the notations and basic estimates. Uniqueness of the solution is also proved in this section. In Sections 3 and 4 existence is established via two different methods. The first one studied in Section 3 uses a fixed point argument for an appropriate mapping and based also on
arguments related to the Snell envelope of processes and optimal stopping. The second approximation, given in Section 4, is constructed by penalization of the constraint $Y\geq L$. Finally, in Section 5, we show that, provided the problem is formulated within a Markovian framework, the solution of the reflected BSDE provides a probabilistic representation for the unique viscosity solution of an obstacle problem for a nonlinear parabolic partial differential equation with obstacle or variational inequality. 

\section{Notations, setting of the problem and preliminary results}

\noindent Let $(\Omega ,\mathcal{F},P)$ be a fixed probability space on which is defined a standard $d$-dimensional Brownian motion $B=(B_{t})_{t\leq T}$ whose natural filtration is $({\cal F}_{t}^{0}:=\sigma \{B_{s},s\leq t\})_{t\leq T}$. We denote by $({\cal F}_{t})_{t\leq T}$ the completed filtration of $({\cal F}_{t}^{0})_{t\leq T}$ with the $P$-null sets of $ \mathcal{F}$, therefore $({\cal F}_{t})_{t\leq T}$ satisfies the usual conditions, i.e. it is right continuous and complete.

From now on stochastic processes will be defined for $t \in [0,T]$, where $T$ is a positive real constant, and will takes their values in $\R^n$ for some positive integer $n$. Finally for any $x,x' \in \R^{k}$, $|x|$ denotes the Euclidean norm of $x$ and $\langle x,x' \rangle$ the inner product.
\medskip

Next for any real constant $p\in ]1,2[$, let:

$(i)$ $\spas^{p}(\R^{n})$ be the set of $\R^n$-valued, adapted and continuous processes $\left\{ X_{t} \right\}_{t \in [0,T]}$ such that:
$$\|X \|_{\spas^{p}} = \E \left[ \sup_{t \in [0,T]} |X_{t}|^{p} \right]^{\frac{1}{p}} < + \infty.$$
The space $\spas^{p}(\R^{n})$ endowed with the norm $\| . \|_{\spas^{p}}$ is of Banach type.

$(ii)$ $\spam^{p}(\R^{n})$ be the set of predictable processes $\left\{ Z_{t} \right\}_{t \in [0,T]}$ with values in $\R^{n}$ such that:
$$\|Z \|_{\spam^{p}} = \E \left[ \left( \int_{0}^{T} |Z_{t}|^{2} dt \right)^{p/2} \right]^{ 1/p} < + \infty.$$
Once more $\spam^{p}(\R^{n})$ endowed with this norm $\|Z \|_{\spam^{p}}$ is a Banach space.  
\medskip

Now let $\spaprod^{p}$ be the space $\spas^{p}(\R) \times \spam^{p}(\R^{d})$. Let $\xi$ be an an $\R$-valued and $\tri_T$-measurable random variable and let us consider a random function $f :[0,T] \times \Omega \times \R \times \R^d \to \R$ measurable with respect to ${\cal P} \times \mathcal{B}(\R)\times \mathcal{B}(\R^d)$ where $\cal P $ denotes the $\sigma$-field of progressive subsets of $[0,T]\times \Omega$. Finally let $L:=\{L_t\}_{t \in [0,T]}$ be a continuous progressively measurable $\R$-valued process. On the items $\xi$, $L$ and $f$ we make the following assumptions:
\begin{enumerate}
\item[(H1)] $\xi \in L^{p}(\Omega)$;
\item[(H2)]
\begin{enumerate}\item[(i)] the process $\{f(t,0,0), \ 0 \leq t \leq T\}$ satisfies
    $\displaystyle \E \left( \int_0^T |f(t,0,0)| dt \right)^p < + \infty$;
    \item[(ii)] there exists a constant $\kappa$ such that:
     $$P-a.s.,\,\,|f(t,y,z)-f(t,y',z')| \leq \kappa (|y-y'|+|z-z'|), \,\forall t,y,y',z,z'.$$
    \end{enumerate}
\item[(H3)] the barrier $L$ is s.t. $L_T\leq \xi$ and $L^+ :=L\vee 0\in \spas^{p}(\R)$. 
\end{enumerate}

To begin with let us define the notion of solution of the reflected BSDE associated with the triple $(f,\xi,L)$ which we consider throughout this paper.

\begin{defin}[of $L^p$-solutions]\label{defiLpsol}:
We say that $\{(Y_t,Z_t,K_t), \ 0 \leq t \leq T \}$ is a $L^p$-solution of the reflected BSDE with one continuous lower reflecting barrier $L$, terminal condition $\xi$ and generator $f$ if the followings hold:
\begin{enumerate}
\item $\{(Y_t,Z_t), \ 0 \leq t \leq T \}$ belongs to $\spaprod^{p}$;
\item $K=\{K_t, \ 0 \leq t \leq T\}$ is an adapted continuous non decreasing process s.t. $K_0=0$ and $K_T \in L^p(\Omega)$;
\item $\displaystyle Y_t = \xi + \int_t^T f(s,Y_s,Z_s) ds + K_T - K_t - \int_t^T Z_s dB_s$, $0 \leq t \leq T$ a.s.;
\item $Y_t \geq L_t$, $0 \leq t \leq T$;
\item $\displaystyle \int_0^T (Y_s - L_s) dK_s =0$, P-a.s..
\end{enumerate}
\end{defin}
\medskip

The following corollary whose proof is given in \cite{briand} will be used several times later, therefore for the sake completeness we recall it.
\begin{corol}[Cor.2.3 in \cite{briand}] \label{corol_briand}
Assume that $(Y,Z)\in {\cal B}^p$ is a solution of the following BSDE:
$$Y_t = \xi + \int_t^T \tilde{f}(t,Y_s,Z_s) ds +A_T-A_t- \int_t^T Z_s dB_s,\,\,t\leq T$$
where:

$(i)$ $\tilde{f}$ is a function which satisfies the same assumptions as $f$

$(ii)$ P-$a.s.$ the process $(A_t)_{t\leq T}$ is of bounded variation type.
\medskip

\noindent Then for any $0 \leq t \leq u \leq T$ we have:
\begin{eqnarray*}
&&|Y_t|^p + c(p) \int_t^u |Y_s|^{p-2} \ind_{Y_s \neq 0} |Z_s|^2 ds \\
&& \quad \leq |Y_u|^p+ p \int_t^u |Y_s|^{p-1} \hat{Y}_s dA_s +p \int_t^u |Y_s|^{p-1} \hat{Y}_s \tilde{f}(s,Y_s,Z_s) ds- p \int_t^u |Y_s|^{p-1} \hat{Y}_s Z_s dB_s,
\end{eqnarray*}
where $c(p)=\frac{p(p-1)}{2}$ and $\hat{y} = \frac{y}{|y|} \ind_{y \neq 0}$.
\end{corol}
\medskip

We are now going to focus on uniqueness of the $L^p$-solution of the BSDE associated with $(f,\xi,L)$. However we first provide some estimates, on the one hand, on the bounds of the solution w.r.t. the data $(f,\xi,L)$, and, on the other hand, on the variation of solutions of those BSDEs induced by a variations of the data. Actually we have:

\begin{lem} \label{estimate_on_Z}
Assume that $(Y,Z)\in {\cal B}^p$ is a solution of the following BSDE:
$$Y_t = \xi + \int_t^T f(t,Y_s,Z_s) ds +K_T-K_t- \int_t^T Z_s dB_s,\,\,t\leq T$$
where P-$a.s.$ the process $(K_t)_{t\leq T}$ is continuous non decreasing, with $K_0=0$. There exists a real constant $C_{p,\kappa}$ depending only on $p$ and $\kappa$ such that:
$$\E \left[ \left( \int_0^T |Z_s|^2 ds \right)^{p/2} \right] \leq C_{p,\kappa} \E \left[ \sup_{t \in [0,T]} |Y_t|^p + \left(\int_0^T |f(s,0,0)| ds \right)^p  \right].$$
\end{lem}
\begin{proof}
Let $\al$ be a real constant and for each integer $k$ let us define:
$$\tau_k =\inf  \left\{ t \in [0,T], \ \int_0^t |Z_s|^2 ds \geq k \right\} \wedge T.$$
The sequence $(\tau_k)_{k\geq 0}$ is of stationary type since the process $Z$ belongs to ${\cal M}^p$ and then $\int_0^T |Z_s|^2 ds<\infty,\,\,P-a.s.$. Next using It\^o's formula yields:
\begin{eqnarray*}
&& |Y_0|^2 + \int_0^{\tau_k}  e^{\al s}|Z_s|^2 ds = e^{\al \tau_k}|Y_{\tau_k}|^2 + \int_0^{\tau_k} e^{\al s}Y_s (2f(s,Y_s,Z_s)-\al Y_s) ds \\
&& \qquad + 2 \int_0^{\tau_k} e^{\al s} Y_s dK_s - 2 \int_0^{\tau_k} e^{\al s} Y_s Z_s dB_s \\
&& \leq e^{\al \tau_k}|Y_{\tau_k}|^2 + \int_0^{\tau_k} e^{\al s}\{2|Y_s f(s,0,0)|+2\ka|Y_s|^2 +2\kappa|Y_s||Z_s| -\al |Y|^2_s\}ds\\
&& \qquad + 2 \int_0^{\tau_k} e^{\al s}Y_s dK_s - 2 \int_0^{\tau_k} e^{\al s}Y_s Z_s dB_s \\
&& \leq e^{\al \tau_k}|Y_{\tau_k}|^2 + 2\sup_{s\leq \tau_k}e^{\al s}|Y_s|\times \int_0^{\tau_k} |f(s,0,0)| ds \\
&& \qquad + (2\ka+\kappa \eps^{-1}-\al) \int_0^{\tau_k} e^{\al s}|Y_s|^2 ds + \eps \kappa \int_0^{\tau_k}  e^{\al s}|Z_s|^2 ds \\
&& \qquad + 2 \int_0^{\tau_k} e^{\al s}Y_s dK_s - 2 \int_0^{\tau_k} e^{\al s}Y_s Z_s dB_s
\end{eqnarray*}
for any $\eps >0$. Therefore
\begin{eqnarray*}
&&|Y_0|^2 + (1-\eps \kappa) \int_0^{\tau_k} e^{\al s}|Z_s|^2 ds \leq (e^{\al \tau_k}|Y_{\tau_k}|^2 +(1+\frac{1}{\eps}) \sup_{s\leq \tau_k}e^{2\al s}|Y_s|^2) \\
&& \qquad +\left( \int_0^{\tau_k} |f(s,0,0)| ds \right)^2 + (2\ka+\kappa \eps^{-1}-\al) \int_0^{\tau_k} e^{\al s}|Y_s|^2 ds \\
&& \qquad + \eps K_{\tau_k}^2-2 \int_0^{\tau_k} e^{\al s}Y_s Z_s dB_s.
\end{eqnarray*}
But there exists a constant $C_\kappa$ such that:
\begin{eqnarray*}
K_{\tau_k}^2 & \leq & C_\kappa \left( |Y_0|^2 + |Y_{\tau_k}|^2 + \left(\int_0^{\tau_k} |f(s,0,0)| ds \right)^2 + \int_0^{\tau_k} |Y_s|^2 ds \right. \\
& & \left. + \int_0^{\tau_k} |Z_s|^2 ds + \left| \int_0^{\tau_k} Z_s dB_s \right|^2 \right).
\end{eqnarray*}
Plugging this last inequality in the previous one to get:
\begin{eqnarray*}
&&(1-\eps C_\kappa)|Y_0|^2 + (1-\eps \ka)\int_0^{\tau_k} e^{\al s}|Z_s|^2 ds -\eps C_\kappa \int_0^{\tau_k} |Z_s|^2 ds \\
&& \qquad \leq \{(\eps C_\ka +e^{\al \tau_k})|Y_{\tau_k}|^2 +(1+\frac{1}{\eps}) \sup_{s\leq \tau_k}e^{2\al s}|Y_s|^2\}+\eps C_\kappa \left| \int_0^{\tau_k} Z_s dB_s \right|^2 \\
&& \qquad \qquad + 2 \left| \int_0^{\tau_k} e^{\al s}Y_s Z_s dB_s \right| +(2\ka+\kappa \eps^{-1}-\al) \int_0^{\tau_k} e^{\al s}|Y_s|^2 ds \\
&& \qquad \qquad +(1+\eps C_\kappa) (\int_0^{\tau_k} |f(s,0,0)|ds)^2.
\end{eqnarray*}
Choosing now $\eps$ small enough and $\al$ such that $2\ka+\kappa \eps^{-1}-\al<0$ we obtain:
\begin{eqnarray*}
\E \left( \int_0^{\tau_k} |Z_s|^2 ds \right)^{p/2} & \leq & C(\kappa,p) \left( \left( \int_0^{\tau_k} |f(s,0,0)| ds \right)^p + \sup_{t \in [0,T]} |Y_t|^p  \right) \\
&& + C(\kappa,p) \left| \int_0^{\tau_k} e^{\al s}Y_s Z_s dB_s \right|^{p/2}.
\end{eqnarray*}
Next thanks to BDG's inequality we have:
\begin{eqnarray*}
\E \left[  \left| \int_0^{\tau_k} e^{\al s}Y_s Z_s dB_s \right|^{p/2} \right] & \leq & \bar C_p \E \left[  \left( \int_0^{\tau_n} |Y_s|^2 |Z_s|^2 ds \right)^{p/4} \right]  \\
& \leq & \bar C_p \E \left[  \left( \sup_{t \in [0,T]} |Y_t| \right)^{p/2}  \left( \int_0^{\tau_k} |Z_s|^2 ds \right)^{p/4} \right] \\
& \leq & \frac{\bar C_p^2}{\eta} \E \left[ \sup_{t \in [0,T]} |Y_t|^p \right] + \eta \E  \left( \int_0^{\tau_k} |Z_s|^2 ds \right)^{p/2}.
\end{eqnarray*}
Finally plugging the last inequality in the previous one, choosing $\eta$ small enough and finally using Fatou's Lemma to obtain the desired result.
\end{proof}
\bigskip

We will now establish an estimate for the process $Y$. Actually we have:
\begin{lem} \label{estimate_on_Y}
We keep the notations of Lemma \ref{estimate_on_Z} and we assume moreover that $P$-a.s. $\int_0^T (Y_s-K_s)^+dK_s = 0$. Then there exists a constant $C_{p,\ka}$ such that:
$$\E \sup_{t \in [0,T]} |Y_t|^p \leq C_{\ka,p} \left[ \E |\xi|^p + \E( \int_0^T |f(s,0,0)|ds)^p +  \E \left(\sup_{t \in [0,T]} (L_s^+)^{p} \right) \right].$$
\end{lem}
\begin{proof}
From Corollary \ref{corol_briand} for any $\al \in \R$ and any $0 \leq t \leq u \leq T$ we have:
\begin{eqnarray*}
&& e^{\al p t} | Y_t|^p + c(p) \int_t^u e^{\al ps } |Y_s|^{p-2} \ind_{Y_s \neq 0} |Z_s|^2 ds   \\ \nonumber 
&& \quad \leq e^{\al p u} |Y_u|^p - p \int_t^u \al e^{\al p s } | Y_s|^p ds + p \int_t^u e^{\al p s} |Y_s|^{p-1} \sgn(Y_s) f(s,Y_s,Z_s) ds  \\ \nonumber 
&& \qquad +p \int_t^u e^{\al p s} |Y_s|^{p-1} \sgn(Y_s) dK_s - p \int_t^u e^{\al p s} |Y_s|^{p-1} \sgn(Y_s) Z_s dB_s
\end{eqnarray*}
where $\sgn(y) := \frac{y}{|y|} \ind_{y \neq 0}$. But since $f$ is Lipschitz then we have:
\begin{eqnarray*}
&& e^{\al p t} | Y_t|^p + c(p) \int_t^u e^{\al ps } |Y_s|^{p-2} \ind_{Y_s \neq 0} |Z_s|^2 ds   \\
&& \quad \leq e^{\al p u} |Y_u|^p + p(\ka - \al) \int_t^u e^{\al p s}  |Y_s|^p ds + p \int_t^u e^{\al p s} |Y_s|^{p-1}  |f(s,0,0)| ds  \\
&& \qquad + p\kappa \int_t^u e^{\al p s} |Y_s|^{p-1}  |Z_s| ds  +p \int_t^u e^{\al p s} |Y_s|^{p-1} \sgn(Y_s) dK_s \\
&& \qquad - p \int_t^u e^{\al p s} |Y_s|^{p-1} \sgn(Y_s) Z_s dB_s.
\end{eqnarray*}
As we have
$$p\kappa  |Y_s|^{p-1}|Z_s| \leq \frac{p\kappa^2}{(p-1)} |Y_s|^p + \frac{c(p)}{2} |Y_s|^{p-2} \ind_{Y_s \neq 0} |Z_s|^2,$$
and by Young's inequality it holds true that:
$$p \int_t^u e^{\al ps}|Y_s|^{p-1}|f(s,0,0)|ds \leq (p-1)\gamma^{\frac{p}{p-1}}(\sup_{t\leq s\leq u}|Y_s|^{p})+ \gamma^{-p}(\int_t^ue^{\al ps}|f(s,0,0)|ds)^p $$
for any $\gamma >0$. Then plug the two last inequalities in the previous ones to obtain:
\begin{eqnarray*}
&& e^{\al p t} | Y_t|^p + \frac{c(p)}{2} \int_t^u e^{\al ps } |Y_s|^{p-2} \ind_{Y_s \neq 0} |Z_s|^2 ds   \\
&& \quad \leq e^{\al p u} |Y_u|^p + (p-1)\gamma^{\frac{p}{p-1}}(\sup_{t\leq s\leq u}|Y_s|^{p})+ \gamma^{-p} \left( \int_t^ue^{\al ps}|f(s,0,0)|ds \right)^p  \\
&& \qquad + p\left( \ka + \frac{\kappa^2}{(p-1)}  - \al \right) \int_t^u e^{\al p s }  | Y_s|^p ds \\
&& \qquad  +p \int_t^u e^{\al p s} |Y_s|^{p-1} \sgn(Y_s) dK_s - p \int_t^u e^{\al p s} |Y_t|^{p-1} \sgn(Y_s) Z_s dB_s.
\end{eqnarray*}
Next let us deal with $\int_t^u e^{\al p s} |Y_s|^{p-1} \sgn(Y_s) dK_s$. Indeed the hypothesis related to increments of $K$ and $Y-L$ implies that $dK_s=1_{[Y_s\leq L_s]}dK_s$, for any $s\leq T$. Therefore we have:
$$\int_t^u e^{\al p s} | Y_s|^{p-1} \sgn( Y_s) d K_s = \int_t^u e^{\al p s}  | Y_s|^{p-1} \sgn( Y_s) 1_{[Y_s\leq L_s]}dK_s \leq \int_t^u e^{\al p s} \theta(L_s) dK_s$$
where $\theta:x\in \R \mapsto \theta(x)=|x|^{p-1}\frac{x}{|x|}1_{[x\neq 0]}$ which is actually a non-decreasing function. It follows that:
\begin{eqnarray*}
&&\int_t^u e^{\al p s} |Y_s|^{p-1} \sgn(Y_s) dK_s \leq \int_t^u e^{\al p s} |L_s|^{p-1} \sgn(L_s) dK_s\\
&& \quad \leq \int_t^u e^{\al p s} (L_s^+)^{p-1} dK_s \leq \left(  \sup_{t \in [0,T]} (L_s)^+ \right)^{p-1} \int_t^u e^{\al p s} dK_s \\
&& \quad \leq  \frac{(p-1)}{p} \frac{1}{\eps^{\frac{p}{p-1}}} \left( \sup_{t \in [0,T]} ((L_s)^+)^{p} \right) + \frac{1}{p} \eps^{p} \left( \int_t^u e^{\al p s} dK_s \right)^p.
\end{eqnarray*}
for any $\eps > 0$. Therefore choosing $\al$ such that 
$$\ka + \frac{\kappa^2}{p-1} \leq \al$$
then put $u=T$ and taking expectation to obtain,
\begin{eqnarray} \label{eq1-04-03}
&& e^{\al p t} | Y_t|^p + \frac{c(p)}{2} \int_t^T e^{\al ps } |Y_s|^{p-2} \ind_{Y_s \neq 0} |Z_s|^2 ds \leq e^{\al p T} |\xi|^p +(p-1)\gamma^{\frac{p}{p-1}} (\sup_{t\leq s\leq T}|Y_s|^{p})\\
\nonumber && +  \frac{1}{\gamma^p}(\int_t^T e^{\al p s} |f(s,0,0)|ds)^p + (p-1)\frac{1}{\eps^{\frac{p}{p-1}}} \left( \sup_{t \in [0,T]} (L_s^+)^{p} \right) \\ \nonumber 
&& \qquad+ \eps^{p} \left( \int_t^T e^{\al p s} dK_s \right)^p  - p \int_t^T e^{\al p s} |Y_t|^{p-1} \sgn(Y_s) Z_s dB_s.
\end{eqnarray}
Next we focus on the control of the term $ \int_t^T e^{\al p s} dK_s$. So using the predictable dual projection property (see e.g. \cite{della2}) we have: $\forall t\leq T$,
\begin{eqnarray*}
\E \left[ (K_T-K_t)^p \right] & = & \E \left[ \int_t^T p(K_T-K_s)^{p-1} dK_s \right] = p \E \int_t^T \E \left[ (K_T-K_s)^{p-1} | \tri_s \right] dK_s  \\
& \leq & p \E \int_t^T \left[ \E \left (K_T-K_s) | \tri_s \right) \right]^{p-1} dK_s, \quad \mbox{ since } p\in ]1,2[.
\end{eqnarray*}
The last inequality holds true thanks to Jensen's conditional one. Recall now that
\begin{equation*}
K_T - K_t =  Y_t - \xi - \int_t^T f(s,Y_s,Z_s) ds + \int_t^T Z_s dB_s
\end{equation*}
then
\begin{eqnarray*}
\E \left[ (K_T-K_t)^p \right] & \leq & p \E \int_t^T \left[ \E \left (Y_s-\xi - \int_s^T f(u,Y_u,Z_u) du \bigg| \tri_s \right) \right]^{p-1} dK_s \\
& \leq & p \E \int_t^T \left[ \E \left ( 2 \sup_{u \in [t,T]} |Y_u| + \int_s^T |f(u,Y_u,Z_u)| du \bigg| \tri_s \right) \right]^{p-1} dK_s \\
& \leq & \frac{1}{2} \E \left[ (K_T-K_t)^p \right] \\
&& + C_p \E \sup_{s \in [t,T]} \left[ \E \left ( 2 \sup_{u \in [t,T]} |Y_u| + \int_t^T |f(u,Y_u,Z_u)| du \bigg| \tri_s \right) \right]^{p}.
\end{eqnarray*}
The last inequality is obtained once more through Young's one. Thus using now Doob's maximal inequality to obtain:
\begin{eqnarray*}
\frac{1}{2}\E \left[ (K_T-K_t)^p \right] & \leq & C_p  \sup_{s \in [t,T]} \E \left[ \E \left ( 2\sup_{u \in [t,T]} |Y_u| + \int_t^T |f(u,Y_u,Z_u)| du \bigg| \tri_s \right) \right]^{p} \\
& \leq & \tilde{C}_p \E \left[  \sup_{u \in [t,T]} |Y_u|^p +  \left( \int_t^T |f(u,Y_u,Z_u)| du \right)^p \right] \\
& \leq & \bar{C}_{p,\kappa} \E \left[ \sup_{u \in [t,T]} |Y_u|^p + \left( \int_t^T |f(u,0,0)| du \right)^p + \left( \int_t^T |Z_u| du \right)^p \right].
\end{eqnarray*}
and then by Lemma \ref{estimate_on_Z} we have 
\begin{equation} \label{eq2-04-03}
\E \left[ (K_T-K_t)^p \right] \leq {C}_{p,\kappa} \E \left[ \sup_{u \in [t,T]} |Y_u|^p + \left( \int_t^T |f(u,0,0)| du \right)^p \right].
\end{equation}
Hereafter $C_{\ka,p}$ is a constant which depends on $p$, $\ka$ and possibly $T$ which may change from a line to another. Now the local martingale $(\int_0^t e^{\al p s} |Y_t|^{p-1} \sgn(Y_s) Z_s dB_s)_{t\leq T}$ is actually a martingale, therefore taking expectation in (\ref{eq1-04-03}) and taking into account of (\ref{eq2-04-03}) to obtain:
\begin{eqnarray}\label{est-term-quad}
&&  \frac{c(p)}{2}\E \int_t^T e^{\al ps } |Y_s|^{p-2} \ind_{Y_s \neq 0} |Z_s|^2 ds \leq e^{\al p T} \E |\xi|^p+\frac{p-1}{\eps^{\frac{p}{p-1}}} \E\left( \sup_{t \in [0,T]}
(L_s^+)^{p} \right) \\ \nonumber 
&& \qquad + C_{\ka,p}\left \{(\gamma^{\frac{p}{p-1}}+\eps^p) \E(\sup_{t\leq s\leq T}|Y_s|^{p})+(\frac{1}{\gamma^p}+\eps^p)\E(\int_t^T e^{\al p s} |f(s,0,0)|ds)^p\right\}.
\end{eqnarray}
Next going back to (\ref{eq1-04-03}) taking the supremum and then expectation we get after taking into account (\ref{eq2-04-03}) 
\begin{eqnarray} \label{eq4-04-03}
&&\E \sup_{t \in [0,T]} e^{\al p t} | Y_t|^p+ \frac{c(p)}{2}\E \int_t^T e^{\al ps } |Y_s|^{p-2} \ind_{Y_s \neq 0} |Z_s|^2 ds\leq \\ \nonumber 
&&  e^{\al p T} \E |\xi|^p + C_{p,\kappa}\gamma^{-p}(\E \int_0^T e^{\al p s} |f(s,0,0)|ds)^p + C_{p,\ka}(\gamma^{\frac{p}{p-1}}+ \eps^{p} ) \E \sup_{u \in [t,T]} |Y_u|^p \\ \nonumber 
&& \qquad +\frac{p-1}{\eps^{\frac{p}{p-1}}}\E \left( \sup_{t \in [0,T]} (L_s^+)^{p} \right) + p \E \sup_{t \in [0,T]} \left| \int_0^T e^{\al p s} |Y_t|^{p-1} \sgn(Y_s) Z_s dB_s \right|.
\end{eqnarray}
Next using the BDG inequality we have
\begin{eqnarray*}
&& \E \sup_{t \in [0,T]} \left| \int_0^T e^{\al p s} |Y_t|^{p-1} \sgn(Y_s) Z_s dB_s \right| \leq 2 \E \left( \int_0^T e^{2 \al p s} |Y_t|^{2(p-1)} \ind_{Y_s \neq 0}  |Z^n_s|^2 ds \right)^{1/2} \\
&& \qquad \leq 2 \E \left[ \left( \sup_{t \in [0,T]} e^{\al p t /2} |Y_t|^{p/2} \right) \left( \int_0^T e^{\al p s} |Y_t|^{p-2}  \ind_{Y_s \neq 0} |Z_s|^2 ds  \right)^{1/2} \right] \\
&& \qquad \leq \eta \E \left( \sup_{t \in [0,T]} e^{\al p t } |Y_t|^{p} \right) + \frac{1}{\eta} \E \left( \int_0^T e^{\al p s} |Y_t|^{p-2}  \ind_{Y_s \neq 0} |Z_s|^2 ds  \right).
\end{eqnarray*}
We now plug this inequality in (\ref{eq4-04-03}) and we obtain:
\begin{eqnarray*}
&&\E \sup_{t \in [0,T]} e^{\al p t} | Y_t|^p  \leq e^{\al p T} \E |\xi|^p + C_{p,\kappa} (\gamma^{-p}+\eps ^p )(\E \int_0^T e^{\al p s} |f(s,0,0)|ds)^p  \\ \nonumber
&& \qquad + \frac{p-1}{\eps^{\frac{p}{p-1}}} \E \left( \sup_{t \in [0,T]}  (L_s^+)^{p} \right) + \{C_{p,\ka}(\gamma^{\frac{p}{p-1}}+ \eps^{p}) +p\eta\} \E \sup_{u \in [t,T]} |Y_u|^p \\ \nonumber && \qquad + \frac{p}{\eta} \E \left( \int_0^T e^{\al p s} |Y_t|^{p-2}  \ind_{Y_s \neq 0} |Z_s|^2 ds \right).\\
&& \leq \left( 1+\frac{2p}{c(p)\eta} \right) e^{\al p T} \E |\xi|^p + \left( 1+\frac{2p}{c(p)\eta} \right) \frac{p-1}{\eps^{\frac{p}{p-1}}} \E \left( \sup_{t \in [0,T]}  (L_s^+)^{p} \right)\\
&& \qquad+C_{p,\kappa} (\gamma^{-p}+\eps ^p )\left(1+\frac{2p}{c(p)\eta}\right) \left( \E \int_0^T e^{\al p s} |f(s,0,0)|ds \right)^p \\ \nonumber 
&&  \qquad + \left\{ C_{p,\ka} \left(1+\frac{2p}{\eta c(p)} \right)(\gamma^{\frac{p}{p-1}}+ \eps^{p}) +p\eta  \right\} \E \sup_{u \in [t,T]} |Y_u|^p .
\end{eqnarray*}
Finally it is enough to chose $\eta=\frac{1}{2p}$ and $\gamma$, $\epsilon$ small enough to obtain the desired result.
\end{proof}
\medskip

\begin{lem} \label{estim_apriori1}
Assume that $(f,\xi,L)$ and $(f',\xi',L')$ are two triplets satisfying Assumptions (H). Suppose that $(Y,Z,K)$ is a solution of the RBSDE $(f,\xi,L)$ and $(Y',Z',K')$ is a solution of the RBSDE $(f',\xi',L')$. Let us set:
$$\begin{array}{ccc}
\Delta f= f- f', &\Delta \xi = \xi - \xi' & \Delta L = L-L' \\
\Delta Y= Y- Y', &\Delta Z = Z - Z' & \Delta K = K-K'
\end{array}$$
and assume that $\D L \in L^p([0,T] \times \Prb)$. Then there exists a constant $C$ such that
\begin{eqnarray*}
\E \sup_{t \in [0,T]} |\Delta Y_t|^p & \leq & C \E \left[ |\Delta \xi|^p + \left(\int_0^T |\Delta f(s,Y_s,Z_s)|ds\right)^p  \right] \\
& & \quad + C (\Psi_T)^{1/p} \left[  \E \sup_{t \in [0,T]} |\Delta L_t|^p \right]^{\frac{p-1}{p}},
\end{eqnarray*}
with
\begin{eqnarray*}
\Psi_T & = & \E \left[ |\xi|^p + \left(\int_0^T |f(u,0,0)|du \right)^p + \left(\sup_{t \in [0,T]} (L_t^+)^p \right) \right. \\
& & \qquad \left. + |\xi'|^p + \left(\int_0^T |f'(u,0,0)|du \right)^p + \left(\sup_{t \in [0,T]} ((L'_t)^+)^p \right)\right].
\end{eqnarray*}
\end{lem}
\begin{proof}
Using Corollary \ref{corol_briand}, we have for all $0 \leq t \leq T$:
\begin{eqnarray} \nonumber
&& |\D Y_t|^p + c(p)\int_t^T |\D Y_s|^{p-2} \ind_{\D Y_s \neq 0} |\D Z_s|^2 ds \leq |\D \xi|^p \\ \nonumber && \qquad + p \int_t^T |\D Y_s|^{p-1} \sgn(\D Y_s) (f(s,Y_s,Z_s)-f'(s,Y'_s,Z'_s)) ds \\ \nonumber 
&& \qquad + p \int_t^T \al |\D Y_s|^{p-1} \sgn(\Delta Y_s) d(\Delta K_s) - p \int_t^u |\D Y_s|^{p-1} \sgn(\D Y_s) \D Z_s dB_s\\
\label{comp_ineq1} && \quad \leq |\D \xi|^p + p \kappa \int_t^T |\D Y_s|^{p} ds \\ \nonumber && \qquad + p \kappa \int_t^T |\D Y_s|^{p-1} |\D Z_s| ds + p \int_t^T |\D Y_s|^{p-1} |\D f(s,Y_s,Z_s)| ds \\ \nonumber 
&& \qquad + p \int_t^T |\D Y_s|^{p-1} \sgn(\Delta Y_s) d(\Delta K_s) - p \int_t^T |\D Y_s|^{p-1} \sgn(\D Y_s) \D Z_s dB_s.
\end{eqnarray}
First note that
$$p\kappa |\D Y_s|^{p-1} |\D Z_s| \leq \frac{p \kappa^2}{(p-1)} |\D Y_s|^p + \frac{c(p)}{2} |\D Y_s|^{p-2} \ind_{\D Y_s \neq 0} |\D Z_s|^2.$$
Next if we denote by $\theta$ the function $ (x,a) \mapsto |x-a|^{p-2} \ind_{x \neq a} (x-a)$, we have
$$\begin{array}{l}
\int_t^T |\D Y_s|^{p-1} \sgn(\Delta Y_s) d K_s = \int_t^T |\D Y_s|^{p-1} \sgn(\Delta Y_s) d K_s\\
\qquad\qquad \qquad\qquad =\int_t^T \theta(Y_s,Y'_s) 1_{[Y_s=L_s]}dK_s = \int_t^T \theta(L_s,Y'_s) dK_s.
\end{array}$$ 
In the same way dealing with the other term as previously to obtain:
\begin{eqnarray*}
\int_t^T |\D Y_s|^{p-1} \sgn(\Delta Y_s) d(\Delta K_s) & = & \int_t^T |L_s-Y'_s|^{p-2} \ind_{L_s-Y'_s \neq 0} ( L_s - Y'_s) dK_s \\
&& - \int_t^T | Y_s - L'_s|^{p-2} \ind_{ Y_s - L'_s \neq 0} ( Y_s - L'_s) dK'_s.
\end{eqnarray*}
But for any $x,a\in \R$, the functions $a\in \R\mapsto \theta (x,a)$ and $x\in \R\mapsto \theta (x,a)$ are respectively non-increasing and non-decreasing, therefore:
\begin{eqnarray*} 
&&\int_t^T |\D Y_s|^{p-1} \sgn(\Delta Y_s) d(\Delta K_s)\leq \\ 
&& \int_t^T |\D L_s|^{p-2} \ind_{\D L_s \neq 0} ( \D L_s) dK_s - \int_t^T |\D L_s|^{p-2} \ind_{\D L_s \neq 0} ( \D L_s ) dK'_s \\
&&= \int_t^T |\D L_s|^{p-1} d(\D K_s)
\end{eqnarray*} 
since $Y\geq L$ and $Y'\geq L'$. Thus coming back to (\ref{comp_ineq1}) to get
\begin{eqnarray} \label{comp_ineq2}
&& |\D Y_t|^p + \frac{c(p)}{2} \int_t^T |\D Y_s|^{p-2} \ind_{\D Y_s \neq 0} |\D Z_s|^2 ds \leq |\D \xi|^p + p\int_t^T |\D Y_s|^{p-1} |\D f(s,Y_s,Z_s)| ds  \\ \nonumber 
&& \qquad + \left( p \kappa + \frac{p \kappa^2}{(p-1)} \right)\int_t^T |\D Y_s|^{p} ds  \\
\nonumber && \qquad + p \int_t^T |\D L_s|^{p-1} d(\D K_s) - p \int_t^T |\D Y_s|^{p-1} \sgn(\D Y_s) \D Z_s dB_s.
\end{eqnarray}
On the other hand the process $\left\{\int_0^t |\D Y_t|^{p-1} \sgn(\D Y_s) \D Z_s dB_s \right\}_{0 \leq t \leq T}$ is a martingale thanks to the Burkholder-Davis-Gundy and Young inequalities. With $t=0$ and taking the expectation in (\ref{comp_ineq2}) we have
\begin{eqnarray*}
&& \frac{c(p)}{2} \E \int_0^T |\D Y_s|^{p-2} \ind_{\D Y_s \neq 0} |\D Z_s|^2 ds \leq \E |\D \xi|^p \\
&& \qquad + \left( p \kappa + \frac{p \kappa^2}{(p-1)}  \right) \E \int_0^T |\D Y_s|^{p} ds  \\
&& \qquad + p\E \int_0^T |\D Y_s|^{p-1} |\D f(s,Y_s,Z_s)|  ds + p \E \int_0^T |\D L_s|^{p-1} d(\D K_s).
\end{eqnarray*}
Coming once again back to (\ref{comp_ineq2}), we also have
\begin{eqnarray*}
&& \E |\D Y_t|^p \leq \E |\D \xi|^p + \left( p \kappa + \frac{p \kappa^2}{(p-1)}  \right) \E \int_t^T |\D Y_s|^{p} ds \\
&& \qquad + p\E \int_0^T |\D Y_s|^{p-1} |\D f(s,Y_s,Z_s)|  ds + p \E \int_0^T |\D L_s|^{p-1} d(\Delta K_s).
\end{eqnarray*}
With the Gronwall lemma we conclude that
$$ \E \int_0^T |\D Y_s|^{p} ds \leq C_p \E \left( |\D \xi|^p + \int_0^T |\D Y_s|^{p-1} |\D f(s,Y_s,Z_s)|  ds + \int_0^T |\D L_s|^{p-1} d(\D K_s) \right).$$ 
Now with H\"older's inequality
$$\E \int_0^T |\D L_s|^{p-1} d(\D K_s) \leq \left( \E \sup_{s \in [0,T]}|\Delta L_s|^{p} \right)^{\frac{p-1}{p}} \E \left(|\Delta K_T|^p \right)^{\frac{1}{p}}$$
and since $\E |\Delta K_T|^p  \leq C_p (\E |K_T|^p  + \E |K'_T|^p)$, using inequality (\ref{eq2-04-03}) and Lemma \ref{estimate_on_Y}, we deduce that
$$\E |\Delta K_T|^p \leq C \Psi_T.$$
Therefore we obtain
\begin{eqnarray*}
&& \E \int_0^T |\D Y_s|^{p} ds + \frac{c(p)}{2} \E \int_0^T |\D Y_s|^{p-2} \ind_{\D Y_s \neq 0} |\D Z_s|^2 ds \\
&& \qquad \leq C \E \left( |\D \xi|^p + \int_0^T |\D Y_s|^{p-1} |\D f(s,Y_s,Z_s)| ds \right) + C \left( \E \sup_{s \in [0,T]}|\Delta L_s|^{p} \right)^{\frac{p-1}{p}} \left( \Psi_T \right)^{\frac{1}{p}}.
\end{eqnarray*}
But
\begin{eqnarray*}
\int_0^T |\D Y_s|^{p-1} |\D f(s,Y_s,Z_s)| ds & \leq & \sup_{s\leq T}|\D Y_s|^{p-1}\times \int_0^T  |\D f(s,Y_s,Z_s)| ds \\
& \leq & \rho^{\frac{p}{p-1}} \sup_{s\leq T}|\D Y_s|^{p}+\frac{1}{\rho^p} \left( \int_0^T |\D f(s,Y_s,Z_s)| ds \right)^p
\end{eqnarray*}
for any $\rho>0$. Next with (\ref{comp_ineq2}), BDG inequality, and the two previous inequalities, we obtain after having chosen $\rho$ small enough:
\begin{eqnarray*}
\E \sup_{s \in [0,T]} |\D Y_s|^{p} & \leq & C \E \left( |\D \xi|^p + \left\{\int_0^T  |\D f(s,Y_s,Z_s)|  ds \right\}^p \right)\\
&& + C \left( \E \sup_{s \in [0,T]}|\Delta L_s|^{p} \right)^{\frac{p-1}{p}} \left( \Psi_T \right)^{\frac{1}{p}}.
\end{eqnarray*}
The conclusion of the Lemma follows.
\end{proof}

\begin{thm}
Under the assumptions [H1]-[H3], there is at most one $L^p$-solution for the reflected BSDE associated with $(f,\xi,L)$.
\end{thm}
\begin{proof}
Using the previous Lemma (since $L-L' = 0 \in L^p$), we obtain immediatly $Y=Y'$. Therefore we have also $Z=Z'$ and finally $K=K'$, whence uniqueness of the $L^p$-solution of the reflected BSDE associated with $(f,\xi,L)$.
\end{proof}

\section{Existence via the Snell Envelope Method}

We now focus on the issue of existence. To begin with let us first assume that the function $f$ does not depend on $(y,z)$.
\begin{thm} \label{eqssyz}
The reflected BSDE associated with $(f(t),\xi, L)$ has a unique $L^p$-solution.
\end{thm}
\begin{proof} We are going to proof the existence of a solution in using the Snell envelope of processes. The Snell envelope of a process $X$ which belongs to class [D] is the smallest supermartingale of class [D] which is greater than $X$.

For $t\leq T$ let us set:
$$\tilde Y_t= \underset{\tau \geq t}{\mbox{esssup}} \E \left[ \int_0^\tau f(s)ds +L_\tau \ind_{[\tau <T]}+\xi \ind_{[\tau =T]}|{\cal F}_t \right].$$
Here $\tau$ is a ${\cal F}_t$-stopping time. The processes $\tilde Y$ verifies $\tilde Y_T=\xi$ and is called the Snell envelope of the process which is inside $\mbox{esssup}$.

Since the process $(\int_0^t|f(s)|ds)_{t\leq T}$ and $\xi$ belong to $L^p(\Omega)$ and $L^+$ belongs to $\spas^{p}$, then the process $\tilde Y$ exists and belongs to $\spas^{p}$. Furthermore thanks to Doob-Meyer decomposition there exists an increasing continuous process $(K_t)_{t\leq T}$ which belongs to $\spas^{p}$ ($K_0=0$) and
a martingale $(M_t)_{t\leq T}$ which is also in $\spas^{p}$ (see e.g. \cite{della2}, pp.221) such that:
$$\forall t\leq T, \,\, \tilde Y_t=M_t-K_t.$$ 
Next as $M\in \spas^{p}$ then there exists a process $Z\in \spam^{p}$ such that:
$$ \forall t\leq T, M_t=M_0+\int_0^tZ_sdB_s.$$
Now for $t\leq T$, let us set:
$$Y_t=\tilde Y_t-\int_0^tf(s)ds.$$ 
Therefore the triplet $(Y,Z,K)$ verifies: for any $t\leq T$,
$$Y_t = \xi + \int_t^T f(s)ds + K_T - K_t - \int_t^T Z_s dB_s.$$
Moreover  we obviously have $Y\geq L$. In order to complete the proof it remains to show that $(Y_t-L_t)dK_t=0$ for any $t\leq T$. So let $\tau \leq T$ be a stopping time and let us set $L^\xi_t:=L_t \ind_{[t<T]}+\xi \ind_{[t=T]}$ and $D_\tau$ the following stopping time:
$$D_\tau=\inf \left\{s\geq \tau, \tilde Y_s=\int_0^sf(u)du+L^\xi_s \right\}\wedge T.$$
Since the process $L$ is continuous on $[0,T[$ and may have a positive jump at $T$, then the stopping time $D_\tau$ is optimal after $\tau$. It follows that the process $(\tilde Y_s)_{s\in [\tau,D_\tau]}$ is a martingale and $\tilde Y_{D_\tau}=L^\xi_{D_\tau}+\int_0^{D_\tau}f(s)ds$ (see
e.g. \cite{nekaspproba}, pp.129, pp.143). Henceforth we have $\int_\tau^{D_\tau}(\tilde Y_s-\int_0^sf(u)du-L^\xi_s)dK_s=0$ which
implies that $\int_0^{T}(\tilde Y_s-\int_0^sf(u)du-L^\xi_s)dK_s=0$. If not, by continuity we can find a stopping time $\tau$ such that $\int_\tau^{D_\tau}(\tilde Y_s-\int_0^sf(u)du-L^\xi_s)dK_s>0$, which is absurd. Now the definition of $Y$ implies also that:
$$\int_0^{T}(Y_s-L_s)dK_s=0.$$ 
Thus the proof is complete.
\end{proof}
\bigskip

We now deal with the general case of generator i.e. $f$ depends on $(y,z)$ and is Lipschitz w.r.t. those arguments. So for $(U,V) \in \spaprod^p$ we define $(Y,Z,K) = \Phi(U,V)$ where $(Y,Z)$ is the $L^p$-solution of the BSDE associated with $(f(t,U_t,V_t),\xi,L)$, i.e.,
\begin{eqnarray*} 
&& (Y,Z)\in {\cal B}^p,\,\, K\in \spas^p\\&&Y_t = \xi + \int_t^T f(s,U_s,V_s)ds + K_T - K_t - \int_t^T Z_s dB_s,\,\,t\leq T \\
&&Y_t \geq L_t \mbox{ and }(Y_t-L_t)dK_t=0.
\end{eqnarray*}
The solution of this equation exists and is unique thanks to Theorem \ref{eqssyz}.
\medskip

Now for $(U',V')$ in $\spaprod^p$, we define in the same way $(Y',Z')=\Phi(U',V')$ and $(\D Y, \D Z)$ by $(Y-Y',Z-Z')$, $\D f_s = f(s,U_s,V_s)-f(s,U'_s,V'_s)$.

We are now going to prove that there exists a real constant $\al \in \R$ such that $\Phi$ is a contraction on $\spaprod^p$, equipped with the equivalent norm:
$$\| (Y,Z) \| = \| e^{\al .} Y \|_{\spas^{p}}  + \|e^{\al .} Z \|_{\spam^{p}}.$$
Actually we have:
\begin{lem} \label{control_on_snell_env}
There exists $\al \in \R$ and a constant $C_{\al}$ such that:
\begin{eqnarray}\label{control_y}
\| e^{\al .} \D Y \|_{\spas^{p}} \leq C_{\al} (\| e^{\al .} \D U \|_{\spas^{p}} + \|e^{\al .} \D V \|_{\spam^{p}} ).
\end{eqnarray}
\end{lem}
\begin{proof}
Using Corollary \ref{corol_briand}, we have for all $0 \leq t \leq u \leq T$:
\begin{eqnarray} \label{eq1_27_06}
&& e^{\al p t} |\D Y_t|^p + c(p) \int_t^u e^{\al ps } |\D Y_s|^{p-2} \ind_{\D Y_s \neq 0} |\D Z_s|^2 ds   \\ \nonumber 
&& \quad \leq e^{\al p u} |\D Y_u|^p  + p \int_t^u e^{\al p s} |\D Y_s|^{p-1} \sgn(\D Y_s) \D f_s ds - p \int_t^u \al e^{\al p s }  |\D Y_s|^p ds \\ \nonumber 
&& \qquad + p \int_t^u e^{\al p s} |\D Y_s|^{p-1} \sgn(\D Y_s) d(\D K_s) - p \int_t^u e^{\al p s} |\D Y_s|^{p-1} \sgn(\D Y_s) \D Z_s dB_s.
\end{eqnarray}
Now for $\eps > 0$, using Young's inequality 
\begin{eqnarray*}
&&\int_t^u e^{\al p s} |\D Y_s|^{p-1} \sgn(\D Y_s) \D f_s ds \leq \int_t^u e^{\al p s} \left( \eps^{-\frac{p}{p-1}} \frac{p-1}{p} |\D Y_s|^{p} + \frac{\eps^p}{p} |\D f_s|^p \right) ds \\
&& \quad \leq \eps^{-\frac{p}{p-1}} \frac{p-1}{p} \int_t^u e^{\al p s} |\D Y_s|^{p} ds + \frac{\kappa^p 2^{p-1}\eps^p}{p} \int_t^u e^{\al p s} \left( |\D U_s|^{p} +  |\D V_s|^{p} \right)  ds \\
&& \quad  \leq \eps^{-\frac{p}{p-1}} \frac{p-1}{p} \int_t^u e^{\al p s}  |\D Y_s|^{p} ds + \frac{\kappa^p 2^{p-1}T\eps^p}{p} \left[ \left( \sup_{s \in [t,u]}  e^{\al p s}  |\D Y_s|^{p} \right) \right. \\
&& \qquad \left. + \left(  \int_t^u   e^{2 \al s} |\D V_s|^{2} ds
\right)^{p/2} \right].
\end{eqnarray*}
Moreover using Fatou's Lemma
\begin{eqnarray*}
&& \int_t^u e^{\al p s} |\D Y_s|^{p-1}  \sgn(\D Y_s) d(\D K_s) = \int_t^u e^{\al p s}|\D Y_s|^{p-2} \ind_{\D Y_s \neq 0} ( Y_s - L_s) dK_s \\
&& \qquad + \int_t^u e^{\al p s}|\D Y_s|^{p-2} \ind_{\D Y_s \neq 0} ( Y'_s - L_s) dK'_s \\
&& \qquad - \int_t^u e^{\al p s}|\D Y_s|^{p-2} \ind_{\D Y_s \neq 0} ( Y_s - L_s) dK'_s -  \int_t^u e^{\al p s}|\D Y_s|^{p-2} \ind_{\D Y_s \neq 0} ( Y'_s - L_s) dK_s \\
&& \quad \leq \int_t^u e^{\al p s}|\D Y_s|^{p-2} \ind_{\D Y_s \neq
0} ( Y_s - L_s) dK_s
+ \int_t^u e^{\al p s}|\D Y_s|^{p-2} \ind_{\D Y_s \neq 0} ( Y'_s - L_s) dK'_s \\
&& \quad =0
\end{eqnarray*}
since $dK_s=\ind_{[Y_s=L_s]}dK_s$ and
$dK'_s=\ind_{[Y'_s=L_s]}dK'_s$, for any $s\in [0,T]$. Coming back to
(\ref{eq1_27_06}) we obtain:
\begin{eqnarray} \label{eq2_27_06}
&& e^{\al p t} |\D Y_t|^p + c(p) \int_t^u e^{\al ps } |\D Y_s|^{p-2}
\ind_{\D Y_s \neq 0} |\D Z_s|^2 ds   \\ \nonumber && \quad \leq
e^{\al p u} |\D Y_u|^p  +  \left(  \eps^{-\frac{p}{p-1}}
\frac{p-1}{p} - p \al \right) \int_t^u e^{\al p s}  |\D Y_s|^{p}  ds
\\ \nonumber && \qquad + \frac{\kappa^p 2^{p-1} T\eps^p}{p} \left[
\left( \sup_{s \in [t,u]} e^{\al p s} |\D Y_s|^{p} \right) + \left(
\int_t^u e^{2 \al s} |\D V_s|^{2} ds \right)^{p/2} \right] \\
\nonumber && \qquad - p \int_t^u e^{\al p s} |\D Y_t|^{p-1} \sgn(\D
Y_s) \D Z_s dB_s.
\end{eqnarray}
But as in the proof of uniqueness, the process 
$$\left\{ M_t = \int_0^t e^{\al p s} |\D Y_t|^{p-1} \sgn(\D Y_s) \D Z_s dB_s \right\}_{0 \leq t \leq T}$$ 
is a uniformly integrable martingale.
Therefore with (\ref{eq2_27_06}), and by choosing $\al$ such that
$\displaystyle \eps^{-\frac{p}{p-1}} \frac{p-1}{p} - p \al \leq 0$,
we obtain:
\begin{eqnarray} \label{eq3_27_06}
&& c(p) \E \left[ \int_0^T e^{\al ps } |\D Y_s|^{p-2} \ind_{\D Y_s
\neq 0} |\D Z_s|^2 ds  \right] \leq \frac{\kappa^p
2^{p-1}T\eps^p}{p} \E \left[ \left( \sup_{s \in [t,u]}  e^{\al p s}
|\D U_s|^{p} \right) \right. \\ \nonumber && \qquad \left. + \left(
\int_t^u e^{2 \al s} |\D V_s|^{2} ds \right)^{p/2} \right]
\end{eqnarray}
and
\begin{eqnarray} \label{eq4_27_06}
\E \left[ \sup_{t \in [0,T]} e^{\al p t} |\D Y_t|^p \right] & \leq &
\frac{\kappa^p 2^{p-1}T\eps^p}{p} \E \left[ \left( \sup_{s \in
[t,u]} e^{\al p s}  |\D U_s|^{p} \right) \right. \\ \nonumber &&
\left. + \left( \int_t^u e^{2 \al s} |\D V_s|^{2} ds \right)^{p/2}
\right] + p \E \left[ \langle M,M \rangle_T^{1/2} \right] .
\end{eqnarray}
For the last inequality we have made use of BDG's one. But
\begin{eqnarray*}
&& \E \left[ \langle M,M \rangle_T^{1/2} \right]  \leq \E \left[ \left(\sup_{t \in [0,T]}  e^{\al p t/2} |\D Y_t|^{p/2} \right) \left(  \int_0^T e^{\al ps } |\D Y_s|^{p-2} \ind_{\D Y_s \neq 0} |\D Z_s|^2  \right)^{1/2}  \right] \\
&&  \leq \frac{1}{2p}  \E \left[ \sup_{t \in [0,T]}  e^{\al p t} |\D
Y_t|^{p} \right] + \frac{p}{2} \E \left[ \int_0^T e^{\al ps } |\D
Y_s|^{p-2} \ind_{\D Y_s \neq 0} |\D Z_s|^2 \right].
\end{eqnarray*}
Plugging now that inequality in (\ref{eq3_27_06}) and
(\ref{eq4_27_06}) to obtain:
\begin{eqnarray} \label{eq2_26_09}
&&\frac{1}{2} \E \left[ \sup_{t \in [0,T]} e^{\al p t} |\D Y_t|^p \right] \leq \frac{\kappa^p  2^{p-1}T\eps^p}{p} \E \left[ \left(
\sup_{s \in [0,T]}  e^{\al p s}  |\D U_s|^{p} \right) \right. \\ \nonumber
&& \hspace{7cm} \left. + \left( \int_t^u   e^{2 \al s} |\D V_s|^{2} ds \right)^{p/2} \right] \\ \nonumber 
&&\qquad  + \frac{p\kappa^p  2^{p-1}T\eps^p}{2c(p)} \E \left[
\left( \sup_{s \in [0,T]}  e^{\al p s}  |\D U_s|^{p} \right) +
\left( \int_t^u   e^{2 \al s} |\D V_s|^{2} ds \right)^{p/2} \right].
\end{eqnarray}
Finally it is enough to choose
$$C_{\al} =  \frac{2\kappa^p  2^{p-1}T\eps^p}{p} \left( 1 + \frac{p2}{2 c(p)}
\right) \quad \mbox{and} \quad \eps^{-\frac{p}{p-1}} \frac{p-1}{p}
\leq p \al.$$Thus the proof is complete.
\end{proof}
\bigskip

We next focus on the same estimate for $\D Z$.

\begin{lem} \label{control_on_mart_part}
There exists $\beta \in \R$ and a constant $C'_{\beta}$ such that
\begin{eqnarray}\label{control_z}\|e^{\beta .} \D Z \|_{\spam^{p}} \leq C'_{\beta}  (\|
e^{\frac{\beta}{2} .} \D U \|_{\spas^{p}} +  \|e^{\beta .} \D V
\|_{\spam^{p}} ).\end{eqnarray}
\end{lem}
\begin{proof}
For each integer $n\geq 1$ let us set: $$\tau_n = \inf \left\{ t \in
[0,T], \ \int_0^t |\D Z_s |^2 ds \right\} \wedge T.$$ Therefore
using It\^o's formula leads to
\begin{eqnarray*}
|\D Y_0|^2 + \int_0^{\tau_n} e^{\beta s} |\D Z_s|^2 ds & = & e^{\beta \tau_n} |\D Y_{\tau_n}|^2 + 2 \int_0^{\tau_n} e^{\beta s} \D Y_s \D f_s ds - \beta \int_0^{\tau_n} e^{\beta s} |\D Y_s|^2 ds \\
&+ & 2 \int_0^{\tau_n} e^{\beta s} \D Y_s d(\D K_s) - 2
\int_0^{\tau_n} e^{\beta s} \D Y_s \D Z_s dB_s.
\end{eqnarray*}
But for any $s \in [0,T]$, $\D Y_s d(\D K_s) \leq 0$ a.s. On the
other hand since $f$ is a Lipschitz function then for every $\nu
> 0$
\begin{eqnarray*}
|\D Y_0|^2 + \int_0^{\tau_n} e^{\beta s} |\D Z_s|^2 ds & \leq & e^{\beta \tau_n} |\D Y_{\tau_n}|^2 + \left(  \frac{\kappa^2}{\nu}  - \beta \right)  \int_0^{\tau_n} e^{\beta s} |\D Y_s|^2 ds \\
& + & \nu \int_0^{\tau_n} e^{\beta s} (|\D U_s|^2 + |\D V_s|^2) ds -
2 \int_0^{\tau_n} e^{\beta s} \D Y_s \D Z_s dB_s.
\end{eqnarray*}
Now if $\frac{\kappa^2}{\nu} \leq \beta$ we obtain:
\begin{eqnarray*}
\int_0^{\tau_n} e^{\beta s} |\D Z_s|^2 ds & \leq & e^{\beta \tau_n} |\D Y_{\tau_n}|^2 + (\nu T) \left[ \sup_{t \in [0,T] } e^{\beta t} |\D U_t|^2 \right]  \\
& + & \nu  \int_0^{\tau_n} e^{\beta s} |\D V_s|^2 ds  + 2 \left|
\int_0^{\tau_n} e^{\beta s} \D Y_s \D Z_s dB_s \right|.
\end{eqnarray*}
It follows that
\begin{eqnarray*}
\left(  \int_0^{\tau_n} e^{\beta s} |\D Z_s|^2 ds \right)^{p/2} & \leq & 2^{(p-1)} \left\{ e^{\beta \tau_n p/2} |\D Y_{\tau_n}|^p + (\nu T)^{p/2} \left[ \sup_{t \in [0,T] } e^{\beta t p/2} |\D U_t|^p \right] \right. \\
& + & \left.  \nu^{p/2}  \left( \int_0^{\tau_n} e^{\beta s} |\D
V_s|^2 ds \right)^{p/2} + 2^{p/2} \left| \int_0^{\tau_n} e^{\beta s}
\D Y_s \D Z_s dB_s \right|^{p/2} \right\}.
\end{eqnarray*}
But by the BDG inequality we have:
\begin{eqnarray*}
&& \E \left| \int_0^{\tau_n} e^{\beta s} \D Y_s \D Z_s dB_s \right|^{p/2} \leq \bar{c}_p\E \left[ \left( \int_0^{\tau_n} e^{2 \beta s} |\D Y_s|^2 |\D Z_s|^2 ds \right)^{p/4} \right] \\
&& \qquad \leq \bar{c}_p2  2^{3p/2} \E \left[ \sup_{t \in [0,T]} e^{
\beta s p/2} |\D Y_s|^p \right] + 2^{-3p/2} \E \left[ \left(
\int_0^T e^{\beta s} |\D Z_s|^2 ds \right)^{p/2} \right] .
\end{eqnarray*}
Therefore plugging this inequality in the previous one to obtain:
\begin{eqnarray*}
&& \frac{1}{2} \E \left( \int_0^{\tau_n} e^{\beta s} |\D Z_s|^2 ds  \right)^{p/2} \leq  2^{3p-1}
\bar{c}_p2 \E \left[ \sup_{t \in [0,T]}  e^{\beta s p/2} |\D Y_s|^p \right]+\\
&& \qquad 2^{(p-1)} \E\left\{ e^{\beta \tau_n p/2} |\D Y_{\tau_n}|^p
+ (\nu T)^{p/2} \left[ \sup_{t \in [0,T] } e^{\beta t p/2} |\D
U_t|^p \right] \right. \\ 
&& \hspace{6cm} \left. + \nu^{p/2}  \left( \int_0^{\tau_n} e^{\beta s} |\D
V_s|^2 ds \right)^{p/2} \right\}.
\end{eqnarray*}
Next using Fatou's Lemma yields:
\begin{eqnarray*}
\frac{1}{2} \E \left(  \int_0^{T} e^{\beta s} |\D Z_s|^2 ds \right)^{p/2} & \leq &  \nu^{p/2} 2^{(p-1)} \E \left\{ T^{p/2} \left[ \sup_{t \in [0,T] } e^{\beta t p/2} |\D U_t|^p \right] \right. \\
& + & \left.  \left( \int_0^{T} e^{\beta s} |\D V_s|^2 ds
\right)^{p/2} \right\} + 2^{3p-1} \bar{c}_p2 \E \left[ \sup_{t \in
[0,T]}  e^{ \beta p/2 s} |\D Y_s|^p \right]
\end{eqnarray*}
Finally choosing $\beta$ great enough (recall that $\beta> 0$) and
using Lemma \ref{control_on_snell_env}, to obtain :
\begin{equation} \label{eq3_26_09}
\E \left(  \int_0^{T} e^{\beta s} |\D Z_s|^2 ds \right)^{p/2} \leq
C'_{\beta} \E \left[ \sup_{t \in [0,T] } e^{\beta t p} |\D U_t|^p +
\left( \int_0^{T} e^{\beta s} |\D V_s|^2 ds \right)^{p/2} \right]
\end{equation}
with
$$C'_{\beta} = 2^{3p} \ \bar{c}_p2 \ C_{\beta} +  \nu^{p/2} 2^{p} \max(T^{p/2},1) \ \mbox{and} \ \frac{\kappa^2}{\nu} \leq \beta.$$
\end{proof}

As a result of Lemmas 1 $\&$ 2 we have:
\begin{prop}\label{contract}
There exist two constants $\gamma$ and  $\mathcal{C} < 1$ such that:
\begin{eqnarray*}
&& \E \left[ \sup_{t \in [0,T] } e^{\gamma t p} |\D Y_t|^p \right] + \E \left(  \int_0^{T} e^{\gamma s} |\D Z_s|^2 ds \right)^{p/2} \\
&& \quad \leq \mathcal{C} \left\{ \E \left[ \sup_{t \in [0,T] }
e^{\gamma t p} |\D U_t|^p \right] + \E \left( \int_0^{T} e^{\gamma
s} |\D V_s|^2 ds \right)^{p/2} \right\}.
\end{eqnarray*}
\end{prop}
\begin{proof}
Recall that in the proofs of Lemmas \ref{control_on_snell_env} and
\ref{control_on_mart_part} we have required that the constants
$\eps$, $\al$, $\nu$ and $\beta$ should satisfy:
$$\eps^{-\frac{p}{p-1}} \frac{p-1}{p} \leq  p \al, \ \frac{\kappa^2}{\nu} \leq \beta$$
$$C_{\al} =  \frac{2\kappa^p  2^{p-1}T\eps^p}{p} \left( 1 + \frac{p2}{4 c(p)} \right),$$
$$C'_{\beta} = 2^{3p} \ \bar{c}_p2 \ C_{\al} +  \nu^{p/2} 2^{p} \max(T^{p/2},1).$$
So we can choose $\eps$ and $\nu$ in such a way that $C_{\al } <
1/2$ and $C'_{\beta} < 1/2$. Therefore the inequalities
(\ref{control_y}) and (\ref{control_z}) still valid if we replace
$\al$ and $\beta$ with $\gamma=\max\{\al,\beta\}$. Also it is enough
to choose $\mathcal{C} = C_{\gamma} + C'_{\gamma} < 1$ and the claim
is proved.
\end{proof}
\bigskip

We now give the main result of this section.

\begin{thm} Under [H1]-[H3], there exists a unique $L^p$-solution for the reflected BSDE associated with $(f(t,y,z),\xi,L)$, i.e., there exists a triple of processes $(Y,Z,K)$ such that:
\begin{eqnarray*}&& Y\in \spas^p, Z\in \spam^p, K\in \spas^p \mbox{ non-decreasing and }K_0=0\\
&&Y_t = \xi + \int_t^{T} f(r,Y_r,Z_r) dr + K_T - K_t - \int_t^T Z_r dB_r, \,\,\forall t\leq T;\\
&&Y\geq L \mbox{ and }(Y_t-K_t)dK_t=0,\,\,\forall t\leq T.
\end{eqnarray*}
\end{thm}
\begin{proof}
Thanks to Proposition \ref{contract}, the mapping $\Phi$ is a
contraction in the Banach space ${\cal B}^p$ endowed with the
equivalent norm $$\|(Y,Z)\|^p_{\gamma,p}=\E\left[(\sup_{t\leq
T}e^{\gamma s}|Y_s|)^p\right ]+\E\left [\left(\int_0^Te^{\gamma
s}||Z_s|^2ds\right)^{p/2}\right].$$ Therefore it has a fixed point
$(Y,Z)$ which in combination with the associated $K$ is the unique
solution of the reflected BSDE associated with $(f(t,y,z,),\xi,L)$.
\end{proof}

\section{Existence via Penalization} \label{section_penaliz}
%---------------------

We are going now to deal with the issue of existence of the
$L^p$-solution for the reflected BSDE associated with
$(f(t,y,z),\xi,L)$ in using the penalization method. Actually for
$n\geq 1$ let us consider $(Y^n,Z^n) \in \spaprod^p$ the unique
solution of the following BSDE:
$$\forall t \in [0,T], \, Y^n_t = \xi + \int_t^T f(s,Y^n_s,Z^n_s) ds + n \int_t^T (Y^n_s - L_s)^- ds - \int_t^T Z^n_s dB_s.$$
Indeed thanks to the result by Briand et al. \cite{briand}, this
solution exists and is unique. Next let us define $K^n$ by:
$$\forall t \in [0,T], \ K^n_t = n \int_0^t (Y^n_s - L_s)^- ds .$$

We first give some estimates for the processes $Y^n$, $Z^n$ and
$K^n$. Actually we have:
\begin{prop} \label{apriori_Lp_estimate}
There exists some constants $\al \in \R$  and $C$ which do not
depend on $n$ such that:
$$\E \left[ \sup_{t \in [0,T]} \left( e^{\al p s} |Y^n_s|^p \right) + \left( \int_0^T  e^{2 \al s} |Z^n_s|^2 \right)^{p/2} + |K^n_T|^p
\right] \leq C.$$
\end{prop}
\begin{proof} It is obtained thanks to Lemmas \ref{estimate_on_Z} and \ref{estimate_on_Y}, in combination with inequality
(\ref{eq2-04-03}). Indeed we have
$$\int_0^T (Y^n_s - L_s)^+ dK^n_s = n\int_0^T (Y^n_s - L_s)^+ (Y^n_s - L_s)^- ds = 0.$$
\end{proof}

We are now going to show that the sequence of processes $(Y^n,Z^n,K^n)$ converges to the $L^p$-solution of the BSDE associated with $(f(t,y,z),\xi,L)$.
\bigskip

First thanks to comparison we have $Y^n\leq Y^{n+1}$, for any $n\geq
0$. Note that here comparison can be obtained as usual (see e.g.
\cite{KPQ97}).Therefore there exists a $\cal P$-measurable process
$Y$ such that for any $t\leq T$, $Y_t=\lim_{n\rightarrow
\infty}\nearrow Y^n_t$. Additionally thanks to Fatou's Lemma we have
$$\E[|Y_t|^p]<\infty,\,\,\forall t\leq T$$
since $\E\sup_{t\leq T}|Y^n_t|^p\leq C$,$\forall n\geq 0$.
\medskip

We now focus on the following:
\begin{lem}: $P$-a.s., $Y\geq L$, $Y\in \spas^p$ and  $\E[(\sup_{s\leq T}(Y^n_s-L_s)^-)^p]\rightarrow 0$ as $n\rightarrow \infty$.
\end{lem}
\begin{proof}
For any $n\geq 0$ and $t\leq T$, we have:
$$Y^n_t-Y^0_t=\int_t^T\{a^n(s)(Y^n_s-Y^0_s)+b^n(s)(Z^n_s-Z^0_s)\}ds+(K^n_T-K^n_t) -\int_t^T(Z^n_s-Z^0_s)dB_s$$
where the processes $(a^n(s))_{s\leq T}$ and $(b^n(s))_{s\leq T}$ are ${\cal P}$-measurable and uniformly bounded by the Lipschitz constant of $f$. But through Proposition
\ref{apriori_Lp_estimate}, there exists a constant $C$ such that:
$$ \forall n\geq 0, \ \E \left[ \int_0^T |a^n_s(Y^n_s-Y^0_s)+b^n_s(Z^n_s-Z^0_s)|^p ds \right]+\E \left[ \int_0^T|Z^n_s-Z^0_s|^p ds \right] \leq C.$$ 
Therefore there exist subsequences and processes $(g_t)_{t\leq T}$ and $(z_t)_{t\leq T}$ which are the weak limits in $L^p_\R([0,T]\times \Omega,dt\otimes dP,{\cal P})$ of
$(g^n_s:=(a^n_s(Y^n_s-Y^0_s)+b^n_s(Z^n_s-Z^0_s))_{s\leq T}$ and $(z^n_s:=Z^n_s-Z^0_s)_{s\leq T}$. Henceforth for any stopping time $\tau \leq T$, the following weak convergence holds :
$$\int_0^\tau z^n_sdB_s\rightarrow \int_0^\tau z_sdB_s \mbox{ and }\int_0^\tau g^n_sds\rightarrow \int_0^\tau g_sdB_s.$$ 
But
$$K^n_\tau=-(Y^n_\tau-Y^0_\tau)+(Y^n_0-Y^0_0)-\int_0^\tau g^n_sds+\int_0^\tau z^n_sdB_s$$
thus we have also the weak convergence
$$K^n_\tau\rightarrow K_\tau:=-(Y_\tau-Y^0_\tau)+(Y_0-Y^0_0)-\int_0^\tau g_sds+\int_0^\tau z_sdB_s \mbox{ and }\E(K_\tau)^p<\infty.$$ 
Now for any stopping times $\sigma\leq \tau\leq T$ it holds true that $K_\sigma\leq K_\tau$ since $K^n_\sigma\leq K^n_\tau$. It follows that the process $(K_t)_{t\leq T}$ is non-decreasing. Additionally we have $E[(K_T)^p]\leq \liminf_{n\rightarrow \infty}E[(K^n_T)^p]\leq C$ since the norm is $lsc$ with respect to the weak topology. Henceforth thanks to the monotonic limit of S.Peng (\cite{peng}, Lemma 2.2, pp.481) the processes $Y-Y^0$ and $K$ are RCLL and so is $Y$ since $Y^0$ is continuous.
\medskip

Next from $E[(K^n_T)^p]\leq C$ for any $n\geq 0$ we deduce, in taking the limit as $n\rightarrow \infty$, that:
$$\E\int_0^T(L_s-Y_s)^-ds]=0$$
and then $P$-$a.s.$, $Y_t\geq L_t$ for any $t<T$. As $\xi\geq L_T$ it follows that $Y\geq L$. Using now Dini's theorem and the Lebesgue dominated convergence one to obtain:
$$\E[(\sup_{s\leq T}(L_s-Y^n_s)^-)^p]\rightarrow 0\mbox{ as }n\rightarrow \infty.$$
Now for any $t\leq T$ we have, 
$$ Y^0_t\leq Y^n_t\leq \sup_{t\leq t}(L_t-Y^n_t)^-+(L^\xi_t)^+.$$
Taking the limit as $n\rightarrow \infty$ to get that $Y\in \spas^p$ since $Y^0$ and $L^\xi$ (see Section 3 for its definition) are so.
\end{proof}
\begin{rem}
Note that the process $Y$ is rcll and its jumps are negative since $\Delta Y=-\Delta K\leq 0$.
\end{rem}
\begin{thm}
The sequence of processes $((Y^n,Z^n,K^n))_{n\geq 0}$ converges to the $L^p$-solution $(Y,Z,K)$ of the BSDE (\ref{RBSDE}) associated with $(f(t,y,z),\xi,L)$.
\end{thm}
\begin{proof}
For $k\geq 0$, let us set:
$$\tau_k:=\inf\{t\geq 0,Y_t\geq k \mbox{ or }Y^0_t\leq -k \mbox{ or }|L_t|\geq k\}\wedge T.$$
The sequence of stopping times $(\tau_k)_{k\geq 0}$ is increasing,
of stationary type converging to $T$ since the process $Y$ is RCLL
and  $Y^0$, $L$ continuous. Additionally we have:$$\max\{\sup_{t\leq
\tau_k}|L_t|,\sup_{t\leq \tau_k}|Y_t|,\sup_{t\leq
\tau_k}|Y^n_t|\}\leq k$$ since $L$ and $Y^0$ are continuous, $Y$ has
only negative jumps and finally $Y^0\leq Y^n\leq Y$. Next for any
$k\geq 0$ and $n\geq 0$ we have:
\begin{eqnarray}\label{eqloc}
Y^n_{t\wedge \tau_k}=Y^n_{\tau_k}+\int_{t\wedge
\tau_k}^{\tau_k}f(s,Y^n_s,Z^n_s)ds+K^n_{\tau_k}-K^n_{t\wedge
\tau_k}-\int_{t\wedge \tau_k}^{\tau_k}Z^n_sdB_s,\,\,\forall t\leq
T.\end{eqnarray} Then for any $n,m$ and $t\leq T$, it holds true
that:
\begin{eqnarray*}
&&Y^n_{t\wedge \tau_k}-Y^m_{t\wedge
\tau_k}=(Y^n_{\tau_k}-Y^m_{\tau_k})+\int_{t\wedge
\tau_k}^{\tau_k}\{f(s,Y^n_s,Z^n_s)-f(s,Y^m_s,Z^m_s)\}ds\\&&\qquad\qquad\qquad\qquad\quad
+(K^n_{\tau_k}-K^m_{\tau_k})-(K^n_{t\wedge \tau_k}-K^m_{t\wedge
\tau_k})-\int_{t\wedge
\tau_k}^{\tau_k}(Z^n_s-Z^m_s)dB_s\\
&&=(Y^n_{\tau_k}-Y^m_{\tau_k})+\int_{t\wedge
\tau_k}^{\tau_k}\{a^{n,m}(s)(Y^n_s-Y^m_s)+b^{n,m}(s)(Z^n_s-Z^m_s)\}ds\\&&\qquad\qquad\qquad\qquad\quad
+(K^n_{\tau_k}-K^m_{\tau_k})-(K^n_{t\wedge \tau_k}-K^m_{t\wedge
\tau_k})-\int_{t\wedge \tau_k}^{\tau_k}(Z^n_s-Z^m_s)dB_s.
\end{eqnarray*}
where $a^{n,m}$ and $b^{n,m}$ are $\cal P$-measurable processes
uniformly bounded by the Lipschitz constant of $f$. Using now
It\^o's formula to obtain: \begin{eqnarray*}&&(Y^n_{t\wedge
\tau_k}-Y^m_{t\wedge \tau_k})2+\int_{t\wedge
\tau_k}^{\tau_k}|Z^n_s-Z^m_s|^2ds=(Y^n_{\tau_k}-Y^m_{\tau_k})2\\&&
\qquad\qquad +2\int_{t\wedge
\tau_k}^{\tau_k}\{a^{n,m}(s)(Y^n_s-Y^m_s)2+b^{n,m}(s)(Y^n_s-Y^m_s)(Z^n_s-Z^m_s)\}ds\\&&\qquad\qquad
+2\int_{t\wedge
\tau_k}^{\tau_k}(Y^n_s-Y^m_s)(dK^n_t-dK^m_t)-2\int_{t\wedge
\tau_k}^{\tau_k}(Y^n_s-Y^m_s)(Z^n_s-Z^m_s)dB_s.
\end{eqnarray*}
Next using localization and then taking expectation to obtain:
\begin{eqnarray*}
&&\E \int_{0}^{t\wedge\tau_k}|Z^n_s-Z^m_s|^2ds\\
&&\quad \leq \E(Y^n_{\tau_k}-Y^m_{\tau_k})2+ C\E\int_{t\wedge \tau_k}^{\tau_k}(Y^n_s-Y^m_s)^2ds + 2\E\int_{t\wedge\tau_k}^{\tau_k}(Y^n_s-Y^m_s)(dK^n_t-dK^m_t)\\
&&\quad  \leq \E(Y^n_{\tau_k}-Y^m_{\tau_k})2+C\E\int_{t\wedge
\tau_k}^{\tau_k}(Y^n_s-Y^m_s)^2ds
+2\E\{(K^n_{\tau_k})^p\}^{1/p}\E\{(\sup_{t\leq
\tau_k}(L_t-Y^n_t)^+)^q\}^{1/q}\\
&& \qquad\qquad\qquad\qquad
+2\E\{(K^m_{\tau_k})^p\}^{1/p}\E\{(\sup_{t\leq
\tau_k}(L_t-Y^m_t)^+)^q\}^{1/q}
\end{eqnarray*}
where $q$ is the conjugate of $p$. Next using dominated convergence
theorem and Proposition \ref{apriori_Lp_estimate} to deduce that:
$$\E\int_{0}^{\tau_k}|Z^n_s-Z^m_s|^2ds\rightarrow 0\mbox{ as }n,m\rightarrow \infty.$$
Now thanks to Lemma \ref{estimate_on_Z}, there exists a constant $C$
such that
$$\E\{\int_0^T|Z^n_s|^pds\}\leq C.$$
Therefore there exists a subsequence and a $\cal P$-measurable
process $Z$ which is the weak limit of that subsequence in
$L^p_\R([0,T]\times \Omega,dt\otimes dP,{\cal P})$. It follows that
for any $k\geq 0$ we have:
$$ \lim_{n\rightarrow \infty}\E \left[\int_0^{\tau_k}|Z^n_s-Z_s|^pds\right]= 0.$$
Further we can argue as in \cite{KKPPQ97} to obtain that:
$$\E \left[\sup_{t\leq T}(|Y^n_{s\wedge \tau_k}-Y^m_{s\wedge \tau_k}|^2+|K^n_{s\wedge \tau_k}-K^m_{s\wedge \tau_k}|^2)\right
]\rightarrow 0 \mbox{ as }n,m\rightarrow \infty. $$ It follows that
for any $k\geq 0$, the process $(Y_{t\wedge \tau_k})_{t\leq T}$ is
continuous and since $(\tau_k)_{k\geq 0}$ is of stationary type then
$Y$ is also a continuous process. Moreover thanks to Dini's theorem
and monotonic convergence theorem we have:
$$E[\sup_{s\leq T}|Y^n_s-Y_s|^p]\rightarrow 0 \mbox{ as }n\rightarrow \infty.$$
Finally for any $t\leq T$, let us set:
$$K_t=Y_0-Y_t-\int_0^tf(s,Y_s,Z_s)ds+\int_0^tZ_sdB_s.$$
Then the process $K$ is continuous, belongs to $\spas^p$ and for
any $k\geq 0$ we have:
$$\E \left[\sup_{t\leq T}|K^n_{s\wedge \tau_k}-K_{s\wedge \tau_k}|^p\right ]\rightarrow 0 \mbox{ as }n\rightarrow \infty.$$
As $K^n$ is increasing and the sequence $(\tau_k)_k$ is of
stationary type then $K$ is also increasing. Now going back to
(\ref{eqloc}) taking the limit as $n\rightarrow \infty$ to obtain
that:
\begin{eqnarray}\label{eqloc2}
Y_{t\wedge \tau_k}=Y_{\tau_k}+\int_{t\wedge \tau_k}^{\tau_k}f(s,Y_s,Z_s)ds+K_{\tau_k}-K_{t\wedge \tau_k}-\int_{t\wedge \tau_k}^{\tau_k}Z_sdB_s,\,\,\forall t\leq
T.\end{eqnarray} 
Additionally we can argue as in \cite {KKPPQ97} to obtain that:
$$ \int_0^{T\wedge \tau_k}(Y_s-L_s)dK_s=0.$$
Taking now the limit w.r.t. $k$ in the two previous equalities to obtain that:
\begin{eqnarray*}
Y_{t}=\xi+\int_{t}^{T}f(s,Y_s,Z_s)ds+K_{T}-K_{t}-\int_{t}^{T}Z_sdB_s,\,\,\forall t\leq T.
\end{eqnarray*}
and 
$$ \int_0^{T}(Y_s-L_s)dK_s=0.$$ 
Henceforth the process $(Y,Z,K)$ is the unique $L^p$-solution of the BSDE associated with $(f(t,y,z),L,\xi)$.
\end{proof}
\begin{rem}
We have also the following result whose proof is classical and then we omit it:
$$\lim_{n \to + \infty} \E \left[  \sup_{t \in [0,T]} |Y^n_t -  Y_t|^p + \left( \int_0^T |Z^n_t - Z_t|^2 dt \right)^{p/2} + \sup_{t \in [0,T]} |K^n_t - K_t|^p  \right] = 0.$$
\end{rem}

\section{Viscosity solutions}
%---------------------

Let $b : \R_+ \times \R^d \to \R^d$, $\sigma : \R_+ \times \R^d \to \R^{d \times d}$ be two globally Lipschitz functions and let us consider the following SDE:
$$dX_t = b(t,X_t)dt + \sigma(t,X_t) dB_t, t\leq T.$$
We denote by $(X^{t,x}_s)_{s \geq t}$ the unique solution of the previous SDE starting from $x$ at time $s=t$. Now we are given three continuous functions:
$$g : \R^d \to \R, \qquad f:[0,T] \times \R^d \times \R \times \R^d \to \R, \qquad h : [0,T] \times \R \to \R$$
such that $f$ is Lipschitz w.r.t. $y$ and $z$, uniformly w.r.t. $t$ and $x$. Moreover there exists $p \in ]1,2[$ s.t. for every $(t,x) \in [0,T] \times \R^d$,
\begin{equation}\label{cond_growth_inf}
\E \int_0^T \left( |f(s,X^{t,x}_s,0,0)|^p + |h(s,X^{t,x}_s)|^p \right) ds + \E |g(X_T^{t,x})|^p < + \infty.
\end{equation}
As said in the introduction, this condition relaxes the standard polynomial growth assumption and will enlarge the class of solutions. It is satisfied if for example $\sigma \sigma^*$ is uniformly elliptic, i.e., there exists $\lambda > 0$ such that 
$$\forall (t,x) \in [0,T] \times \R^d, \quad \forall \zeta \in \R^d \setminus \{ 0\}, \quad \zeta (\sigma \sigma^*)(t,x) \zeta^* \geq \lambda |\zeta|^2,$$
and if for some constant $A > 0$ (depending on $b$ and $\sigma$, and $T$) such that uniformly w.r.t. $t \in [0,T]$
\begin{equation} \label{growth_cond1}
\lim_{|x| \to + \infty} \left[ |f(t,x,0,0)|+ |g(x)| + |h(t,x)| \right] \exp( -A (\ln |x|)2) =0.
\end{equation}
The geometric Brownian motion (or Black-Scholes model in finance) is an example for which such a conditions are satisfied. Note that if we have stronger conditions on $b$ or $\sigma$, we can have weaker growth ones on $f$, $g$ and $h$. 

From now on we assume that $1<p<2$ and that for every $(t,x) \in [0,T] \times \R^d$ let us define $(Y^{t,x}_s,Z^{t,s}_x,K^{t,x}_s)_{s \in [t,T]}$ the unique solution of the reflected BSDE 
$$Y^{t,x}_s = g(X^{t,x}_T) + \int_t^T f(u,X^{t,x}_u, Y^{t,x}_u,Z^{t,x}_u) du + K^{t,x}_T - K^{t,x}_t - \int_t^T Z^{t,x}_u dB_u$$
with
$$a.s. \ \forall s \in [t,T], \ h(s,X^{t,x}_s) \leq Y^{t,x}_s.$$
Moreover on $[0,t]$, we set $Y^{t,x}_s = Y^{t,x}_t$, $Z^{t,x}_s=K^{t,x}_s=0$.

For every $(t,x)$, we will show that $Y^{t,x}_t$ is deterministic and we define a function
\begin{equation} \label{define_sol}
u(t,x) = Y^{t,x}_t.
\end{equation}
In a first part we will prove that $u$ is continuous and is a viscosity solution of the following obstacle problem 
\begin{eqnarray} \nonumber
&& \min \bigg[ u(t,x) - h(t,x), \\ \label{obs_pb} 
&& \qquad -\frac{\partial u}{\partial t}(t,x) - \mathcal{L} u(t,x) - f(t,x,u(t,x),\sigma(t,x) \nabla u(t,x)) \bigg] = 0 \\ \nonumber 
&& \hspace{10cm} (t,x) \in [0,T[ \times \R^d,\\ \nonumber 
&& u(T,x) = g(x), \quad x \in \R^d,
\end{eqnarray}
with the second order partial differential operator
$$\mathcal{L} = \frac{1}{2} \sum_{i,j=1}^d ((\sigma \sigma^*)(t,x))_{i,j} \frac{\partial2 }{\partial x_i \partial x_j} + \sum_{i=1}^d (b(t,x))_i \frac{\partial}{\partial x_i}.$$
In a second part we will prove that this is the unique continuous viscosity solution that belongs to some class of functions. However first let us recall the following definitions:
\begin{defin}
Let $u$ a function that belongs to $C([0,T]) \times \R^d$. It is said to be:

$(i)$ a viscosity subsolution of (\ref{obs_pb}) if $u(T,x) \leq g(x)$, $x \in \R^d$, and for any function $\phi \in C^{1,2}((0,T) \times \R^d)$, if $u-\phi$ has a local maximum at $(t,x)$ then 
$$\min(u(t,x)-h(t,x),-\frac{\partial \phi}{\partial t} - \mathcal{L} \phi (t,x) - f(t,x,u(t,x),\sigma \nabla \phi (t,x))) \leq 0.$$

$(ii)$  a viscosity supersolution of (\ref{obs_pb}) if $u(T,x) \geq g(x)$, $x \in \R^d$, and for any function $\phi \in C^{1,2}((0,T) \times \R^d)$, if $u-\phi$ has a local minimum at $(t,x)$ then
$$\min(u(t,x)-h(t,x),-\frac{\partial \phi}{\partial t} - \mathcal{L} \phi (t,x) - f(t,x,u(t,x),\sigma \nabla \phi (t,x))) \geq 0.$$

$(iii)$ a viscosity solution of (\ref{obs_pb}) if it is both a viscosity sub- and supersolution. 
\end{defin}

\subsection{Continuity and viscosity solution}
%----------------------

We have the following result:
\begin{prop}
For every $(t,x)$, $Y^{t,x}_t$ is deterministic and the function
\begin{equation*}
u(t,x) = Y^{t,x}_t
\end{equation*}
is continuous and satisfies
\begin{equation} \label{cond_infinity}
\lim_{|x| \to + \infty} |u(t,x)| \exp( -A (\ln |x|)^2) =0.
\end{equation}
\end{prop}
\begin{proof}
It suffices to show that whenever $(t_n,x_n) \to (t,x)$,
\begin{equation} \label{continuity1}
\E \left( \sup_{s \in [0,T]} |Y^{t_n,x_n}_s - Y^{t,x}_s|^p \right)
\to 0.
\end{equation}
Indeed, this will show that $(s,t,x) \mapsto Y^{t,x}_s$ is $L^p$
continuous, and so is $(t,x) \mapsto Y^{t,x}_t$. But $Y^{t,x}_t$ is
deterministic, since it is $\tri_t^t$ measurable. Recall that
$\{\tri_s^t, \ t \leq s \leq T\}$ denotes the natural filtration of
the Brownian motion $\{B_s - B_t , \ t \leq s \leq T\}$ argumented
with the $\Prb$ null sets of $\tri$. Now (\ref{continuity1}) is a
consequence of Lemma \ref{estim_apriori1} and the following
convergences
$$\E \left| g(X^{t,x}_T) - g(X^{t_n,x_n}_T) \right|^p \to 0,$$
$$\E \left( \sup_{s \in [0,T]} |h(s,X^{t,x}_s)-h(s,X^{t_n,x_n}_s)|^p \right) \to 0,$$
$$\E \left(\int_0^T \left| \ind_{[t,T]}(s) f(s,X^{t,x}_s,Y^{t,x}_s,Z^{t,x}_s) -
\ind_{[t_n,T]}(s) f(s,X^{t_n,x_n}_s,Y^{t,x}_s,Z^{t,x}_s)
\right|ds\right)^p \to 0,$$ which follow from the continuity
assumptions and the growth of $f$, $g$ and $h$.

The condition (\ref{cond_infinity}) follows from Lemma
\ref{estimate_on_Y} and condition (\ref{growth_cond1}).
\end{proof}

In order to prove that $u$ is a viscosity solution of the obstacle
problem (\ref{obs_pb}) we need a comparison result concerning the
solutions of reflected BSDE. For $\tilde{\xi}$, $\tilde{f}$ and
$\tilde{L}$, let us denote by $(\tilde{Y},\tilde{Z},\tilde{K})$ the
unique solution of
$$\tilde{Y}_t = \tilde{\xi} + \int_t^T \tilde{f}(s, \tilde{Y}_s,\tilde{Z}_s) ds +\tilde K_T-\tilde K_t- \int_t^T \tilde{Z}_s dB_s$$
with
$$P-a.s. \ \forall s \in [0,T], \ \tilde{L}_s \leq \tilde{Y}_s \mbox{ and }\int_0^T(\tilde Y_t-\tilde L_t)d\tilde K_t=0.$$
\begin{prop}
If a.s. $\tilde{\xi} \geq \xi$, $\tilde{f} \geq f$ and $\tilde{L}
\geq L$, then a.s. $\tilde{Y}_t \geq Y_t$ for every $t \in [0,T]$.
\end{prop}
\begin{proof}
Let $(\tilde{Y}^n,\tilde{Z}^n)$ and $(Y^n,Z^n)$ be the sequence
constructed by penalization (see Section \ref{section_penaliz}).
Using the standard comparison result for BSDE (see for example
\cite{pardoux}), then for every $n \in \N$, $\tilde{Y}^n_t \geq
Y^n_t$. Since the sequences $\tilde{Y}^n$ and $Y^n$ converge to
resp. $\tilde{Y}$ and $Y$, the conclusion follows.
\end{proof}

\begin{thm} The function $u:(t,x)\mapsto u(t,x)=Y^{t,x}_t$ defined
in (\ref{define_sol}) is a viscosity solution of the obstacle
problem (\ref{obs_pb}).
\end{thm}
\begin{proof}
We are going to use the approximation of the RBSDE (\ref{RBSDE}) by
penalization, which was studied in Section 4. For each $(t,x) \in
[0,T] \times \R^d$, let $(Y^{t,x,n}, Z^{t,x,n})$ denote the solution
of the BSDE $$\begin{array}{l}\forall s \in [t,T], \ Y^{t,x,n}_s =
g(X^{t,x}_T)+ \int_s^T f(u,Y^{t,x,n}_u,Z^{t,x,n}_u) du
\\\qquad\qquad\qquad\qquad\qquad\qquad\qquad+ n \int_s^T (Y^{t,x,n}_u - h(u,X^{t,x}_u))^- du -
\int_s^T Z^{t,x,n}_u dB_u.\end{array}$$ From \cite{pardoux} it is
known that
$$u_n(t,x) = Y^{t,x,n}_t, \quad (t,x) \in [0,T]\times \R^d,$$
is the viscosity solution of the parabolic PDE: for every $(t,x) \in
[0,T[ \times \R^d$
$$\frac{\partial u_n}{\partial t} (t,x) + \mathcal{L} u_n(t,x) + f_n(t,x,u_n(t,x),\sigma(t,x) \nabla u_n(t,x)) = 0$$
with
$$f_n(t,x,y,z) = f(t,x,y,z) + n(y-h(t,x))^-$$
and $u_n(T,x) = g(x)$ for each $x \in \R^d$. To be more precise, we have to say that in \cite{pardoux} $p$ is supposed to be egal to 2. But with straightforward modifications in the proof of Theorem 3.2 in \cite{pardoux}, we can show that  the result holds also for $1 < p < 2$.

However, from the results of the previous section, for each $(t,x) \in [0,T] \times \R^d$,
$$u_n(t,x) \uparrow u(t,x), \quad \mbox{as } n \to + \infty.$$
Since $u_n$ and $u$ are continuous, it follows from Dini's theorem that the above convergence is uniform on compacts.

We now show that $u$ is a subsolution of (\ref{obs_pb}). Let $(t,x)$ be a point at which $u(t,x) > h(t,x)$, and let $\phi$ be a $C^{1,2}$ function such that $u-\phi$ has a maximum at point $(t,x)$. From Lemma 6.1 in \cite{crandall}, there exists sequences $n_j \to + \infty$, $(t_j,x_j) \to (t,x)$ such that
$$\frac{\partial \phi}{\partial t} (t_j,x_j) + \mathcal{L} \phi(t_j,x_j) + f_{n_j}(t_j,x_j,u_{n_j}(t_j,x_j),\sigma(t_j,x_j) \nabla \phi(t_j,x_j)) \leq 0.$$
From the assumption that $u(t,x) > h(t,x)$ and the uniform convergence of $u_n$, it follows that for $j$ large enough $u_{n_j}(t_j,x_j) > h(t_j,x_j)$. Therefore in taking the limit as $j\to + \infty$, the above inequality yields:
$$\frac{\partial \phi}{\partial t} (t,x) + \mathcal{L} \phi(t,x) + f(t,x,u(t,x),\sigma(t,x) \nabla \phi(t,x)) \leq 0.$$
and we have proved that $u$ is a subsolution of (\ref{obs_pb}).

The same arguments (with converse inequalities) prove that $u$ is also a supersolution of (\ref{obs_pb}). 
\end{proof}

\subsection{Uniqueness of the solution}
%------------------

In order to establish the uniqueness of the solution of equation
result (\ref{define_sol}), we need to impose the following
additional assumption on $f$. For each $R > 0$, there exists a
continuous function $m_R : \R_+ \to \R_+$ such that $m_R(0) = 0$ and
\begin{equation} \label{lip_cond_wrt_x}
|f(t,x,r,p) - f(t,y,r,p)| \leq m_R(|x-y|(1+|p|)),
\end{equation}
for all $t \in [0,T]$, $|x| \leq R$, $|y| \leq R$, $|r| \leq R$, and
$p \in \R^d$.

\begin{thm} \label{thm_uniq_visc_sol}
Under the above assumptions and (\ref{lip_cond_wrt_x}), the obstacle
problem (\ref{obs_pb}) has at most one solution satisfying
(\ref{cond_infinity}).
\end{thm}
The proof is similar to the uniqueness proof given in \cite{BBP97}.
We just have to take into account the obstacle $h$.

Let $u$ and $v$ be two viscosity solutions of (\ref{obs_pb}). As in
\cite{BBP97}, the proof consists in two steps. We first show that
$u-v$ and $v-u$ are subsolutions of a specific partial differential
equation, then we build a suitable sequence of smooth supersolutions
of this equation to show that $|u-v| = 0$ in $[0,T] \times \R^d$.
Hereafter we denote by $\|.\|$ the sup norm in $\R^d$.
\begin{lem}
Let $u$ be a subsolution and $v$ a supersolution of (\ref{obs_pb}).
Then the function $w = u-v$ is a viscosity subsolution of
\begin{equation} \label{eq_diff_sub_super}
\min(w,-\frac{\partial w}{\partial t} - \mathcal{L} w - \kappa \|w\|
- \kappa \|\sigma \nabla w\|)=0,
\end{equation}
where $\kappa$ is the Lipschitz constant of $f$ in $(y,z)$.
\end{lem}
\begin{proof}
First remark that $(u-v)(T,x) \leq 0$. Next let $\phi \in
C^{1,2}((0,T)\times \R^d)$ and let $(t_0,x_0) \in (0,T) \times \R^d$
be a strict global maximum point of $w-\phi$ and we suppose that
$w(t_0,x_0) > 0$. Our aim is to prove that at $(t_0,x_0)$
$$-\frac{\partial \phi}{\partial t} - \mathcal{L} \phi - \kappa |w| - \kappa |\sigma \nabla \phi| \leq 0.$$
We sketch the proof of Lemma 3.7 in \cite{BBP97}. We introduce the
function
$$\psi_{\eps,\al}(t,x,s,y)= u(t,x) - v(s,y) - \frac{|x-y|^2}{\eps^2} - \frac{|t-s|^2}{\al^2} - \phi(t,x),$$
where $\eps$ and $\al$ are positive parameters which are devoted to
tend to zero. Since $(t_0,x_0)$ is a strict global maximum point of
$u-v-\phi$, there exists a sequence
$(\bar{t},\bar{x},\bar{s},\bar{y})$ such that
\begin{itemize}
    \item $(\bar{t},\bar{x},\bar{s},\bar{y})$ is a global maximum point of $\psi_{\eps,\al}$ in $([0,T] \times \bar{B_R})2$ where $B_R$ is a ball with a large radius $R$;
    \item $(\bar{t},\bar{x})$, $(\bar{s},\bar{y}) \to (t_0,x_0)$ as $(\eps,\al) \to 0$;
    \item $\frac{|\bar{x}-\bar{y}|^2}{\eps^2}$ and $\frac{|\bar{t}-\bar{s}|^2}{\al^2}$ are bounded and tend to zero when $(\eps,\al) \to 0$.
\end{itemize}
Moreover there exists two symmetric matrices $X$ and $Y$ such that
since $u$ is a subsolution, at point $(\bar{t},\bar{x})$
\begin{equation} \label{eq_u_subsol}
\min(u-h,\bar{a}-\frac{\partial \phi}{\partial t} - \frac{1}{2}
\mbox{Tr } (\sigma \sigma^* X) - \langle b, (\bar{p} + \nabla \phi)
\rangle - f(\bar{t},\bar{x},u,\sigma (\bar{p} + \nabla \phi))) \leq
0,
\end{equation}
and since $v$ is a supersolution, at point $(\bar{s},\bar{y})$
\begin{equation} \label{eq_v_supersol}
\min(v-h,\bar{a} - \frac{1}{2} \mbox{Tr } (\sigma \sigma^* Y) -
\langle b, \bar{p} \rangle - f(\bar{s},\bar{y},v,\sigma \bar{p} ))
\geq 0,
\end{equation}
where $\bar{a} = \frac{2(\bar{t}-\bar{s})}{\al^2}$ and $\bar{p} =
\frac{2(\bar{x}-\bar{y})}{\eps^2}$. Now we want to substract these
inequalities. Then from the Lipschitz continuity of $\sigma$ and $b$
we obtain:
$$\mbox{Tr } (\sigma \sigma^*(\bar{t},\bar{x}) X) - \mbox{Tr}
(\sigma \sigma^*(\bar{s},\bar{y}) Y) \leq C
\frac{|\bar{x}-\bar{y}|^2 + |\bar{t}-\bar{s}|^2}{\eps^2} + \mbox{Tr
} (\sigma \sigma^*(\bar{t},\bar{x}) D^2 \phi(\bar{t},\bar{x}));$$
and
$$|\langle b(\bar{t},\bar{x}), \bar{p} \rangle - \langle b(\bar{s},\bar{y}), \bar{p} \rangle| \leq
C \frac{|\bar{x}-\bar{y}|^2 + |\bar{t}-\bar{s}|^2}{\eps^2}.$$ We now
consider the difference between the nonlinear terms:
\begin{eqnarray*}
&& f(\bar{t},\bar{x},u(\bar{t},\bar{x}),\sigma(\bar{t},\bar{x}) (\bar{p} + \nabla \phi(\bar{t},\bar{x}))) - f(\bar{s},\bar{y},v(\bar{s},\bar{y}),\sigma(\bar{s},\bar{y}) \bar{p} ) \\
&& \quad \leq \rho_{\eps} (|\bar{t}-\bar{s}|) + m(|\bar{x}-\bar{y}|(1+|\bar{p} \sigma(\bar{s},\bar{y})|)) + \kappa |u(\bar{t},\bar{x})-v(\bar{s},\bar{y})| \\
&& \qquad \qquad \qquad \qquad \qquad \qquad + \kappa
|\bar{p}(\sigma(\bar{t},\bar{x}) - \sigma(\bar{s},\bar{y})) +
\sigma(\bar{t},\bar{x}) \nabla \phi(\bar{t},\bar{x})|.
\end{eqnarray*}
Note that $$|\bar{p}(\sigma(\bar{t},\bar{x}) -
\sigma(\bar{s},\bar{y}))| \leq C \frac{|\bar{x}-\bar{y}|^2 +
|\bar{t}-\bar{s}|^2}{\eps^2}$$ and
$$|\bar{x}-\bar{y}||\bar{p} \sigma(\bar{s},\bar{y})| \leq C \frac{|\bar{x}-\bar{y}|^2}{\eps^2}.$$

Now let us go back to (\ref{eq_u_subsol}) and (\ref{eq_v_supersol}).
We claim that $u(\bar{t},\bar{x}) - h(\bar{t},\bar{x}) > 0$ (taking
a subsequence if necessary). If not, there exists a subsequence such
that $u(\bar{t},\bar{x}) - h(\bar{t},\bar{x}) \leq 0$. Passing to
the limit we get $u(t_0,x_0) - h(t_0, x_0) \leq 0$. But from the
assumption  $u(t_0,x_0) - v(t_0, x_0) > 0$, we deduce that $0 \geq
u(t_0,x_0) - h(t_0, x_0) > v(t_0,x_0) - h(t_0, x_0)$. Therefore we
have $v(\bar{s},\bar{y}) - h(\bar{s},\bar{y}) < 0$, which leads to a
contradiction with (\ref{eq_v_supersol}). Henceforth
(\ref{eq_u_subsol}) becomes
$$\bar{a}-\frac{\partial \phi}{\partial t} - \frac{1}{2} \mbox{Tr } (\sigma \sigma^* X) - \langle b, (\bar{p} + \nabla \phi) \rangle - f(\bar{t},\bar{x},u,\sigma (\bar{p} + \nabla \phi)) \leq 0.$$
Thus we obtain
\begin{eqnarray*}
&& -\frac{\partial \phi}{\partial t} - \frac{1}{2} \mbox{Tr } (\sigma \sigma^* X) - \langle b, (\bar{p} + \nabla \phi) \rangle - f(\bar{t},\bar{x},u,\sigma (\bar{p} + \nabla \phi)) \\
&& \qquad + \frac{1}{2} \mbox{Tr } (\sigma \sigma^* Y) + \langle b,
(\bar{p} + \nabla \phi) \rangle + f(\bar{s},\bar{y},v,\sigma \bar{p}
)) \leq 0.
\end{eqnarray*}
With all previous estimates we get:
\begin{eqnarray*}
&& -\frac{\partial \phi}{\partial t} - \frac{1}{2} \mbox{Tr } (\sigma \sigma^* X) - \langle b, \nabla \phi \rangle - \kappa |u(\bar{t},\bar{x})-v(\bar{s},\bar{y})| - \kappa |\sigma(\bar{t},\bar{x}) \nabla \phi(\bar{t},\bar{x})| \\
&& \quad \leq \rho_{\eps} (|\bar{t}-\bar{s}|) + C
\frac{|\bar{x}-\bar{y}|^2 + |\bar{t}-\bar{s}|^2}{\eps^2} +
m(|\bar{x}-\bar{y}|(1+|\bar{p} \sigma(\bar{s},\bar{y})|)).
\end{eqnarray*}
Next let first $\al$ goes to zero. Since
$\frac{|\bar{t}-\bar{s}|^2}{\al^2}$ is bounded then
$|\bar{t}-\bar{s}|$ goes to zero. We remove the first term  and the
term $|\bar{t}-\bar{s}|^2$ of the right-hand side above. Then we let
$\eps \to 0$ and since $(\bar{t},\bar{x}) \to (t_0,x_0)$ we finally
have
$$-\frac{\partial \phi}{\partial t} (t_0,x_0) - \mathcal L \phi (t_0,x_0) - \kappa |w(t_0,x_0)| - \kappa |\sigma(t_0,x_0) \nabla \phi(t_0,x_0)| \leq 0.$$
Therefore $w$ is a subsolution of the equation
(\ref{eq_diff_sub_super}).
\end{proof}
\medskip

Recall now Lemma 3.8 in \cite{BBP97}.
\begin{lem}
For any $A > 0$, there exists $C > 0$ such that the function
$$\chi(t,x)= \exp \left[(C(T-t) + A) \psi(x) \right]$$
where
$$\psi(x) = \left[\ln((|x|^2+1)^{1/2})+1 \right]^2,$$
satisfies
$$-\frac{\partial \chi}{\partial t} - \mathcal{L} \chi - \kappa  \chi - \kappa  |\sigma \nabla \chi| > 0, \quad \mbox{in} \ [t_1,T] \times \R^d,$$
where $t_1=T-A/C$.
\end{lem}
\begin{proof}
The proof of this result is given in \cite{BBP97}. Note that here,
w.r.t. the setting of this latter article, we have the same
assumptions without the jump process, i.e. $B \equiv 0$. Now the
function $\chi$ is positive therefore it is a supersolution of
(\ref{obs_pb}) with $f(t,x,y,\sigma z) = \kappa y + \kappa |\sigma
z|$ and $h \equiv 0$.
\end{proof}
\bigskip

The rest of the proof of Theorem \ref{thm_uniq_visc_sol} is the same
as in \cite{BBP97}. For any $\al > 0$,
$$|u(t,x)-v(t,x)| \leq \al \chi(t,x) , \ \mbox{ in } [t_1,T] \times \R^d.$$
The sketch of the proof is the following.
\begin{itemize}
    \item Using the growth condition on $u$ and $v$, $\lim_{n \to + \infty} |u-v|(t,x)/\chi(t,x) = 0$ uniformly for $t \in [t_1,T]$, for some $A >0$. This implies that $|u-v| - \al \chi$ is bounded from above in $[t_1,T] \times \R^d$. Hence
    $$M=\max_{[t_1,T] \times \R^d} (|u-v| - \al \chi)e^{K(T-t)}$$
    is achieved at some point $(t_0,x_0)$. Without loss of generality we may assume that $w(t_0,x_0) = (u-v)(t_0,x_0) > 0$.
    \item This means that the function $w-\phi$ has a global maximum point at $(t_0,x_0)$, where
    $$\phi(t,x) = \al \chi(t,x) + (w-\al \chi)(t_0,x_0) e^{K(t-t_0)}.$$
    We use the fact that $w$ is a subsolution of
    (\ref{eq_diff_sub_super}), i.e.,
    $$-\frac{\partial \phi}{\partial t}(t_0,x_0) - \mathcal{L} \phi(t_0,x_0) - \kappa |w(t_0,x_0)| - \kappa |\sigma(t_0,x_0) \nabla \phi(t_0,x_0)|\leq 0.$$
    But since $w(t_0,x_0)=|w(t_0,x_0)|$, the left-hand side is
        $$\al \left[ -\frac{\partial \chi}{\partial t}(t_0,x_0) - \mathcal{L} \chi(t_0,x_0) - \kappa |\chi(t_0,x_0)| - \kappa |\sigma(t_0,x_0) \nabla \chi(t_0,x_0)|) \right].$$
        \item This leads to a contradiction. Then $t_0=T$ and since $w(T,x)=0$, we have $|w(t,x)| - \al \chi(t,x) \leq 0$
        on $[t_1,T] \times \R^d$. Letting $\al$ tending to zero, we obtain $u=v$ in $[t_1,T] \times \R^d$.
        \item Repeat recursively this argument on $[t_2,t_1]$ with $t_2 = (t_1-A/C)^+$, etc., we obtain $u=v$ on $[0,T] \times \R^d$.
\end{itemize}

\end{document}